\documentclass[12pt,a4paper]{article}
\usepackage{amsmath}
\usepackage{amsthm}
\usepackage{amssymb,enumerate}
\usepackage{amsbsy}
\usepackage{amsfonts}
\usepackage{amstext}
\usepackage{amscd}
\usepackage[dvips]{epsfig}

\numberwithin{equation}{section}
\theoremstyle{plain}
\newtheorem{thm}{Theorem}[section]
\newtheorem{prop}[thm]{Proposition}
\newtheorem{cor}[thm]{Corollary}
\newtheorem{lem}[thm]{Lemma}
\theoremstyle{definition}
\newtheorem{exa}[thm]{Example}

\newtheorem{rem}[thm]{Remark}
\newtheorem{defi}[thm]{Definition}

\newcommand{\real}{\mathbb{R}}

\DeclareMathOperator*{\comp}{\mathbb{C}}
\DeclareMathOperator*{\nat}{\mathbb{N}}
\DeclareMathOperator*{\im}{\text{\normalfont Im}}

\newcommand{\id}{\mathcal{ID}}
\newcommand{\utimes}{\kern0.05em\buildrel{\times}\over{\rule{0em}{0.004em}} \kern-0.9em\cup \kern0.2em}
\newcommand{\putimes}{\mathop{\kern0.05em\buildrel{\times}\over{\rule{0em}{0.0em}} \kern-0.9em\cup \kern0em}}
\newcommand{\hutimes}{\mathop{\kern0.02em\buildrel{\times}\over{\rule{0em}{0.0em}} \kern-0.48em\cup \kern0em}}
\newcommand{\sutimes}{\mathrel{\kern0em \buildrel{\mathsf{x}}\over{\rule{0em}{0.0em}} \kern-0.35em\cup \kern-0.0em}}

\newcommand{\tor}{\mathbb{T}}
\newcommand{\disc}{\mathbb{D}}

\begin{document}
\title{Semigroups related to additive and multiplicative, free and Boolean convolutions}

\author{Octavio Arizmendi\footnote{Supported by funds of R.\ Speicher from the Alfried Krupp von Bohlen und Halbach Stiftung. E-mail: arizmendi@math.uni-sb.de} \\ Universit\"{a}t des Saarlandes, FR $6.1-$Mathematik,\\ 66123 Saarbr\"{u}cken, Germany \\ \\ Takahiro Hasebe\footnote{Supported by Global COE Program at Kyoto University. E-mail: thasebe@math.kyoto-u.ac.jp}\\Graduate School of Science,  Kyoto University,\\  Kyoto 606-8502, Japan }   %Graduate School of Science,  Kyoto University,\\  Kyoto 606-8502, Japan}
\date{}
\maketitle
\begin{abstract}
Belinschi and Nica introduced a composition semigroup of maps on the set of probability measures. Using this semigroup, they introduced a free divisibility indicator, from which one can know quantitatively if a measure is freely infinitely divisible or not. 

In the first half of the paper, we further investigate this indicator: we calculate how the indicator changes with respect to free and Boolean powers; we prove that free and Boolean $1/2$-stable laws have free divisibility indicators equal to infinity;
we derive an upper bound of the indicator in terms of Jacobi parameters. This upper bound is achieved only by free Meixner distributions. We also prove Bo\.zejko's conjecture which says the Boolean powers $\mu^\uplus$, $t\in[0,1],$ of a probability measure $\mu$ are freely infinitely divisible if the measure $\mu$ is freely infinitely divisible.
%Moreover, the free L\'{e}vy-Khintchine formula for the resulting probability measure can be calculated in terms of a Boolean L\'{e}vy-Khintchine formula.

In the other half of this paper, we introduce an analogous composition semigroup for multiplicative convolutions and define free divisibility indicators for these convolutions. Moreover, we prove that a probability measure on the unit circle is freely infinitely divisible relative to the free multiplicative convolution if and only if the indicator is not less than one. We also prove how the multiplicative divisibility indicator changes under free and Boolean powers and then the multiplicative analogue of Bo\.zejko's conjecture. We include an appendix, where the Cauchy distributions and point measures are shown to be the only fixed points of the Boolean-to-free Bercovici-Pata bijection. 

\end{abstract}



\section{Introduction}
The class of infinitely divisible distributions has been a central theme in probability theory because they appear as the laws of L\'evy processes. The same is true for free probability, with  L\'evy processes replaced by free L\'evy processes~\cite{Biane}. 
Infinitely divisible distributions in free probability also arise from the eigenvalue distributions of infinitely divisible random matrices \cite{Ben05,Cab05}. 
Important distributions among them are Wigner's semicircle law and the free Poisson law that respectively appear in the eigenvalue distributions of Gaussian unitary ensembles and Wishart matrices. We will develop the theory of free infinite divisibility in this paper.   

For probability measures $\mu, \nu$ on $\real$, let us denote by $\mu \boxplus \nu$ the free additive convolution. $\mu \boxplus \nu$ is the distribution of $X + Y$, where $X$ and $Y$ are self-adjoint free random variables with distributions $\mu$ and $\nu$, respectively.
A probability measure $\mu$ is said to be freely infinitely divisible if for any $n \in \nat$, there exists $\mu_n$ such that $\mu = \mu_n \boxplus \cdots \boxplus \mu_n = \mu_n^{\boxplus n}$, the free convolution of $\mu_n$ by $n$ times.
A striking property of free convolution is that for any $t \geq 1$ and any probability measure $\mu$, the convolution power $\mu^{\boxplus t}$ exists as a probability measure.
This is in contrast to probability theory, because the usual convolution $\mu^{\ast t}$ is not  necessarily defined even for $t \geq 1$, unless $\mu$ is $\ast$-infinitely divisible; the reader is referred to page 2 of \cite{S-V}. Hence, in free probability, the half line $(0,\infty)$ for time parameter is simply divided into two connected components (intervals) by checking whether $\mu^{\boxplus t}$ exists or not. The boundary of these distinct intervals is given by the quantity $\widetilde{\phi}(\mu):= \inf \{t> 0: \mu^{\boxplus t} \text{~exists as a probability measure}\}$. A probability measure $\mu$ is freely infinitely divisible if and only if $\widetilde{\phi}(\mu)=0$.

Another famous convolution arising from non-commutative random variables is a Boolean additive convolution $\mu \uplus \nu$, which is defined as the probability distribution of $X + Y$ for self-adjoint, Boolean independent $X$ and $Y$ with distributions $\mu$ and $\nu$, respectively. It is known that the power $\mu^{\uplus t}$ can be defined for any probability measure $\mu$ and any $t \geq 0$. Thus a difference among classical, free and Boolean convolutions can be observed in existence of convolution powers.  

Using the above properties of free and Boolean convolutions, Belinschi and Nica~\cite{Bel2} introduced a family of maps 
$$
\mathbb{B}_t: \mu \mapsto \left(\mu^{\boxplus (1+t)}\right)^{\uplus \frac{1}{1+t}},~~t\geq 0,  
$$
which turns out to satisfy the semigroup property $\mathbb{B}_t(\mathbb{B}_s(\mu))=\mathbb{B}_{t+s}(\mu)$. 
By looking at the image of the map $\mathbb{B}_t$, the so-called free divisibility indicator $\phi(\mu)$ can be defined for a probability measure $\mu$. 
Unexpectedly, this quantity coincides with $1-\widetilde{\phi}(\mu)$ if $\mu$ is not freely infinitely divisible. 

The explicit calculation of this indicator is expected to be useful to understand the free convolution and free infinite divisibility. In that direction, we prove the relation
\[
\phi(\mu^{\uplus t}) = \frac{\phi(\mu)}{t} \text{~~for~~} t > 0.
\]
%Many results follow from this. %$\phi(\mu)=\sup \{t \geq 0: \mu^{\uplus t} \text{~is freely infinitely divisible}\}$.

%Then a probability measure is freely infinitely divisible if and only if this indicator is greater than or equal to one. Therefore, the detailed study of this indicator

 % In Section \ref{sec2}, we will understand how the free divisibility indicator changes regarding free and Boolean time evolutions. Using this result,
As a byproduct, we find a different characterization of the indicator in terms of Boolean convolution powers: $\phi(\mu)=\sup \{t \geq 0: \mu^{\uplus t} \text{~is freely infinitely divisible}\}$.
This new characterization enables us to interpret the indicator quite naturally in terms of Boolean convolutions.

Another consequence of this relation is that $\mu^{\uplus t}$ is freely infinitely divisible for $0 \leq t \leq 1$ whenever $\mu$ is freely infinitely divisible. This has been conjectured by Bo\.{z}ejko from many calculations~\cite{BW2}.

Also, we give an upper bound of the indicator in terms of Jacobi parameters for any probability measure with a finite fourth moment. Moreover, we show that only free Meixner distributions achieve that upper bound.

Delta and Cauchy distributions have free divisibility indicators equal to infinity. This is because they are the fixed points of the semigroup $\mathbb{B}_t$~\cite{Bel2}. They had been the only known examples whose free divisibility indicators are infinity. We will provide other ones: free stable laws and Boolean strictly stable laws of index $1/2$. However, these examples are not fixed points of the semigroup of homomorphisms. To get these results, we prove that the free divisibility indicator is invariant under shifts and the Boolean convolution with a delta measure. Moreover, we prove in Appendix that the Cauchy distributions and delta measures are the only fixed points of the homomorphisms.  

A multiplicative free convolution $\boxtimes$ and Boolean convolution $\utimes$ and their associated infinite divisibility, which we explain more in Section \ref{sec1}, can be defined on the positive real line and on the unit circle. So, in relation to multiplicative free infinite divisibility, a natural question is if a divisibility indicator and a composition semigroup exist for these convolutions.

The main subject of the second part of this paper is the existence of a counterpart of the semigroup of \cite{Bel2} for the multiplicative convolutions. Remarkably, the same commutation relation used in the additive case is true for multiplicative convolutions.
In contrast to the additive convolutions, some difficulties appear for multiplicative convolutions on the positive real line and on the unit circle. On the unit circle, the problem is the non-uniqueness of convolution powers: neither $\mu^{\boxtimes t}$ nor $\mu^{\hutimes t}$ can be defined uniquely~\cite{Bel3,Fra2}. On the positive real line, a difficulty comes from the fact that $\mu^{\hutimes t}$ cannot be defined for large $t$~\cite{Ber2}. We, however, manage to define composition semigroups and free divisibility indicators for multiplicative convolutions, following the additive case.

Moreover, we prove several results analogous to the additive case. We define free divisibility indicators for these convolutions. For instance, a probability measure on the unit circle is freely infinitely divisible concerning the multiplicative free convolution if and only the indicator is not less than one. We also prove how the multiplicative divisibility indicator changes under free and Boolean powers and then the multiplicative analogue of Bo\.zejko's conjecture.
% We also prove that the Poisson kernel on the unit circle, which corresponds to the Cauchy distribution on the real line, has the multiplicative free divisibility indicator equal to infinity.

%ure $\mu$ is $\boxtimes$-infinitely divisible, then so is $\mu^{\hutimes t}$ for $0 \leq t \leq 1$. the situation for the positive real line differs from that for the unit circle: there are no counterparts of some properties of the additive convolutions.  This is because a Boolean power cannot be defined for a large time in generic cases.
This paper is organized as follows. In Section \ref{sec1} we explain on additive and multiplicative convolutions, both free and Boolean. Section \ref{sec2} is devoted to the study of the additive free divisibility indicator. In Section \ref{sec3} we develop the multiplicative counterparts of composition semigroups.
Combining results of \cite{Bel2} and Section \ref{sec3}, we have simple commutation relations between various pairs of convolutions: the additive free convolution and the additive Boolean one; the multiplicative free convolution and the multiplicative Boolean one. In Section \ref{sec4}, we prove more commutation relations between additive convolutions and multiplicative convolutions. These properties provide new examples of freely infinitely divisible distributions with respect to both multiplicative and additive free convolutions. In the Appendix, the Cauchy distributions and point measures are shown to be the only fixed points of the Boolean-to-free Bercovici-Pata bijection.

\section{Preliminaries} \label{sec1}
\subsection{Additive free convolution}
Let $\mathcal{P}(\real)$ denote the set of the probability measures on $\real$. The upper half-plane and the lower half-plane are respectively denoted as $\comp^+$ and $\comp^-$.
In this article, $G_\mu:\comp^+\rightarrow\comp^-$ denotes the Cauchy transform $G_\mu(z):=\int_{\real}\frac{\mu(dx)}{z-x}$ of a probability measure $\mu$ and
$F_\mu:\comp^+\rightarrow\comp^+$ denotes its reciprocal $\frac{1}{G_\mu(z)}$.
An additive free convolution was introduced by Voiculescu in \cite{Voi86} for compactly supported measures and later extended to all probability measures in \cite{Be-Vo}.
Let $\phi_\mu(z)$ be the Voiculescu transform of $\mu$ defined by
\[
\phi_\mu(z) = F_\mu^{-1}(z)-z
\]
for $z$ in a suitable open set of $\comp^+$. The free convolution $\mu \boxplus \nu$ of probability measures $\mu$ and $\nu$ is characterized by $\phi_{\mu \boxplus \nu}(z) = \phi_\mu (z) + \phi_\nu(z)$ for $z$ in a common domain where $\phi_\mu (z)$ and $\phi_\nu(z)$ are defined. As we mentioned in Introduction, the measure $\mu^{\boxplus t} \in\mathcal{P}(\real)$, satisfying $\phi_{\mu^{\boxplus t}}(z) = t\phi_\mu(z)$,  exists for any $ t \geq 1$ and $\mu \in\mathcal{P}(\real)$.

The free infinite divisibility is characterized in terms of $\phi_\mu$ \cite{Be-Vo}. 
\begin{thm} \label{thmid}The following are equivalent for a probability measure $\mu \in \mathcal{P}(\real)$. 
\begin{enumerate}[\rm(1)]
\item The measure $\mu$ is freely ($\boxplus$- for short) infinitely divisible. 
\item The Voiculescu transform $\phi_\mu$ extends analytically to $\comp^+$ with values in $\comp^- \cup \real$. 
\item  For each $0 < t < \infty$, a probability measure $\mu^{\boxplus t}$ exists satisfying $\phi_{\mu^{\boxplus t}}(z) = t\phi_\mu(z)$. 
\item  The L\'{e}vy-Khintchine representation exists: 
\[
\phi_\mu(z) = \gamma_\mu +\int_{\real}\frac{1+xz}{z-x}\tau_\mu(dx), 
\]
where $\gamma_\mu \in \real$ and $\tau_\mu$ is a non-negative finite Borel measure.
\end{enumerate}
\end{thm}
An infinitely divisible distribution can also be characterized as a weak limit of infinitesimal triangular arrays~\cite{C-G}.

\subsection{Multiplicative free convolutions}

Let $\real_+$, $\disc$ and $\tor$ denote the positive real line $[0,\infty)$, the unit disc $\{z \in \comp: |z| <1 \}$ and the unit circle of the complex plane, respectively. Moreover, let $\mathcal{P}(\real_+)$ and $\mathcal{P}(\tor)$ denote the sets of probability measures on $\real_+$ and $\tor$, respectively.
For probability measures $\mu,\nu$ on $\tor$ and $\real_+$, multiplicative free convolutions $\mu \boxtimes \nu$ on $\tor$ and $\real_+$ were respectively introduced in \cite{Voi87} for compactly supported probability measures. The measure $\mu \boxtimes \nu$ on $\tor$ is the distribution
 of $UV$, where $U$ and $V$ are unitary free random variables with distributions $\mu$ and $\nu$, respectively. Similarly, $\mu \boxtimes \nu$ on $\real_+$ is defined as the
 distribution of $X^{1/2}YX^{1/2}$, where $X$ and $Y$ are positive free random variables with distributions $\mu$ and $\nu$, respectively.
The multiplicative convolution of probability measures on $\real_+$ with non-compact supports was considered in \cite{Be-Vo}.

A probability measure $\mu$ on $\tor$ (resp.,\ $\real_+$) is said to be $\boxtimes$-infinitely divisible if for any $n \in \nat$, there is $\mu_n$ on $\tor$ (resp.,\ $\real_+$) such that $\mu=\mu_n ^{\boxtimes n} = \mu_n \boxtimes \cdots  \boxtimes \mu_n$.

To investigate multiplicative convolutions, an important transform is
\[
\eta_\mu(z)=\frac{\psi_\mu(z)}{1+\psi_\mu(z)}
\]
for $\mu \in \mathcal{P}(\real_+)$ or $\mathcal{P}(\tor)$, where $\psi_\mu(z)$ is a moment generating function defined by $\int_{\text{supp~} \mu} \frac{tz}{1-tz}\mu(dt)$. Note that for $\mu\in \mathcal{P}(\real_+)$, $\eta_\mu = 0$ if and only if $\mu=\delta_0$ and for  $\mu\in \mathcal{P}(\tor)$, $\eta_\mu = 0$ if and only if $\mu=\omega$, the normalized Haar measure on $\tor$.   The transform $\eta_\mu$ is characterized as follows \cite{Bel3}.

\begin{prop}\label{prop0}
(1) Let $\eta: \comp \backslash \real_+ \to \comp$ be an analytic map satisfying $\eta(\overline{z}) = \overline{\eta(z)}$ and $\eta \neq 0$. Then the following properties are equivalent. \\
(1a) $\eta =\eta_\mu$ for a probability measure $\mu \in \mathcal{P}(\real_+)$, $\mu \neq \delta_0$. \\
(1b) $\eta(-0) = 0$ and $\arg \eta(z) \in [\arg z, \pi)$ for any $z \in \comp^+$.

(2) Let $\eta: \disc \to \comp$ be an analytic map. Then the following properties are equivalent. \\
(2a) $\eta =\eta_\mu$ for a probability measure $\mu \in \mathcal{P}(\tor)$. \\
(2b) $|\eta(z)| \leq |z|$ for $z \in \disc$.
\end{prop}

If $\mu \neq \delta_0 \in \mathcal{P}(\real_+)$, the function $\eta_\mu$ is injective in $(-\infty, 0)$, so that 
one can define $\eta_\mu^{-1}(z)$ and 
\[
\Sigma_\mu(z) :=\frac{\eta_\mu^{-1}(z)}{z}
\]
in a suitable open set of $\comp$, which actually contains an interval of the form $(-\alpha, 0)$, $\alpha >0$. The multiplicative free convolution $\boxtimes$ is characterized by the multiplication of $\Sigma$: 
\begin{equation}\label{eq324}
\Sigma_{\mu \boxtimes \nu}(z) = \Sigma_\mu(z) \Sigma_\nu(z)
\end{equation}
for $z$ in an interval $(-\beta, 0)$, provided $\mu \neq \delta_0 \neq \nu$. A measure $\mu \in \mathcal{P}(\real_+)$ is  $\boxtimes$-infinitely divisible if and only if $\Sigma_\mu$ can be written as \cite{Be-Vo}
\begin{equation}\label{eq76}
\Sigma_\mu(z)= \exp\Big(-a_\mu z + b_\mu + \int_{\real_+}\frac{1 +xz}{z-x}\tau_\mu(dx)\Big),
\end{equation}
where $a_\mu \geq 0$, $b_\mu \in \real$ and $\tau_\mu$ is a non-negative finite measure on $\real_+ = [0,\infty)$.

For a measure $\mu$ on $\tor$ with $m_1(\mu):=\int_{\tor} \zeta\, \mu(d\zeta)=\eta_\mu'(0) \neq 0$, the inverse $\eta_\mu^{-1}(z)$ exists in a neighborhood of $0$. The multiplicative convolution $\boxtimes$ is characterized by the same relation (\ref{eq324}), but now the domain for $z$ is a neighborhood of $0$. 
Now a measure $\mu \in \mathcal{P}(\tor)$ is $\boxtimes$-infinitely divisible if and only if $\Sigma_\mu$ can be written as \cite{Be-Vo2}
\begin{equation}\label{idm}
\Sigma_\mu(z)= \gamma_\mu \exp\Big(\int_{\tor}\frac{1 +\zeta z}{1-\zeta z}\tau_\mu(d\zeta)\Big)
\end{equation}
for $z \in \disc$, where $\gamma_\mu \in \tor$ and $\tau_\mu$  is a non-negative finite measure.

If $\mu \in \mathcal{P}(\real_+)$, a convolution power $\mu^{\boxtimes t} \in \mathcal{P}(\real_+)$, satisfying $\Sigma_{\mu^{\boxtimes t}}(z)=(\Sigma_\mu(z))^t$, is well defined for any $\mu$ and any $ t \geq 1$. However this is not true for $\mu \in \mathcal{P}(\tor)$. With additional assumptions that $\eta_\mu$ does not vanish in $\disc$ and $m_1(\mu)=\eta_\mu'(0)\neq 0$, a power $\mu^{\boxtimes t} \in \mathcal{P}(\tor)$ exists for any $\mu$ and any $ t \geq 1$. There is however another problem which will be explained in Section \ref{sec3}. 

$\boxtimes$-infinite divisibility is equivalent to the existence of a weakly continuous convolution semigroup $\mu^{\boxtimes t}$ for $t \geq 0$ with $\mu^{\boxtimes 0} = \delta_1$.  Another characterization of infinitely divisible distributions is found in \cite{B-W} in terms of infinitesimal triangular arrays.

\subsection{Additive and multiplicative Boolean convolutions}
Additive and multiplicative Boolean convolutions on $\real$ and $\tor$ were introduced in \cite{S-W} and \cite{Fra2} respectively.
Let $K_\mu(z)$ be the energy function \cite{S-W} defined by
\[
K_\mu(z)=z-F_\mu(z),~~~z\in\mathbb{C}^+
\]
for $\mu \in \mathcal{P}(\real)$.
The Boolean convolution $\mu \uplus \nu$ is characterized by 
$$
K_{\mu \uplus \nu} (z)=K_\mu(z)+K_\nu(z).
$$
 For any $t > 0$ and any probability measure $\mu$ there exists a probability measure $\mu^{\uplus t}$ such that $K_{\mu^{\uplus t}}(z)=tK_{\mu} (z)$. The L\'{e}vy-Khintchine representation is written as \cite{S-W}
\[
K_\mu(z) = \gamma_\mu +\int_{\real}\frac{1+xz}{z-x}\tau_\mu(dx),
\]
where $\gamma_\mu$ and $\tau_\mu$ satisfy the same conditions as the free case.
%A Boolean convolution semigroup $\mu^{\uplus t}$ can always be defined for any probability measure $\mu \in \mathcal{P}(\real)$.
%\[
%K_{\mu^{\uplus t}}(z) = tK_\mu(z).
%\]
We note that $\eta_\mu$ and $K_\mu$ are related through the formula $\eta_\mu(z) = z K_\mu(\frac{1}{z})$. 

Now an important transform is a meromorphic function $k_\mu(z) := \frac{z}{\eta_\mu(z)}$, defined if $\eta_\mu \neq 0$. 
For $\mu, \nu \in \mathcal{P}(\tor)$, both different from the Haar measure on $\tor$, the multiplicative Boolean convolution $\mu \putimes \nu$ is characterized by
\begin{equation}\label{bo}
k_{\mu \hutimes \nu}(z) = k_\mu (z) k_\nu (z).
\end{equation}
If $\mu, \nu \in \mathcal{P}(\real_+)$, both being different from $\delta_0$, the convolution $\mu \putimes \nu$ is characterized by the same relation (\ref{bo}). 
A probability measure $\mu \in \mathcal{P}(\tor)$ is said to be $\putimes$-infinitely divisible if for any $n \in \nat$, we can find $\mu_n$ such that $\mu=\mu_n^{\hutimes n}$. This is equivalent to the condition that $\frac{1}{k_\mu(z)}$ does not have a zero in $\disc$, and is also equivalent to the existence of the L\'{e}vy-Khintchine formula \cite{Fra2}
\begin{equation}\label{levymb}
k_\mu(z) = \gamma_\mu \exp\Big(\int_{\tor}\frac{1 +\zeta z}{1-\zeta z}\tau_\mu(d\zeta)\Big),
\end{equation}
where $\gamma_\mu$ and $\tau_\mu$ satisfy the same conditions as the free case (\ref{idm}).

Bercovici proved in \cite{Ber2} that the multiplicative Boolean convolution does not preserve $\mathcal{P}(\real_+)$. However, there still exists a Boolean power $\mu^{\hutimes t} \in \mathcal{P}(\real_+)$ for $0 \leq t \leq 1$. 
Results on infinitesimal triangular arrays can be found in \cite{Wang1}.

\section{On the free divisibility indicator}\label{sec2}
A central objective of this paper is a composition semigroup $\{\mathbb{B}_t \}_{t \geq 0}$, introduced by Belinschi and Nica \cite{Bel2}, defined to be
\[
\mathbb{B}_t(\mu) = \Big(\mu^{\boxplus (1+t)}\Big) ^{\uplus\frac{1}{1+t}},~~~\mu\in \mathcal{P}(\real).
\]
In addition to the semigroup property 
\begin{equation}\label{semigroup}
\mathbb{B}_t(\mathbb{B}_s(\mu))=\mathbb{B}_{t+s}(\mu),
\end{equation}
the map $\mathbb{B}_t$ is a homomorphism regarding the multiplicative free convolution $\boxtimes$: 
\begin{equation}\label{homomorphism}
\mathbb{B}_t(\mu \boxtimes \nu) = \mathbb{B}_t (\mu) \boxtimes \mathbb{B}_t(\nu)
\end{equation} 
for probability measures $\mu, \nu$, one of which is supported on $\real_+$. It is known that $\mathbb{B}_1$ coincides with the Bercovici-Pata bijection $\Lambda_B$ from the Boolean convolution to the free one. The reader is referred to \cite{Be-Pa} for the definition of $\Lambda_B$. Let $\phi(\mu)$ denote the free divisibility indicator defined by
\[
\phi(\mu):=\sup \{t \geq 0: \mu \in \mathbb{B}_t(\mathcal{P(\real)}) \}.
\]
For a probability measure $\mu$, Belinschi and Nica proved that a probability measure $\nu$ uniquely exists such that $\mathbb{B}_{\phi(\mu)}(\nu)=\mu$.
Therefore, $\mathbb{B}_t(\mu)$ can be defined as a probability measure for any $t \geq -\phi(\mu)$. This is a natural extension to possibly negative $t$ because the semigroup property (\ref{semigroup})  still holds if $t,s,t+s \geq -\phi(\mu)$. Here we collect properties of the free divisibility indicator~\cite{Bel2}.
\begin{thm}\label{thm001}
(1) $\mu^{\boxplus t}$ exists for $t \geq \max\{1-\phi(\mu),0 \}$. \\
(2) $\mu$ is $\boxplus$-infinitely divisible if and only if $\phi(\mu) \geq 1$. \\
(3) $\phi(\mathbb{B}_t(\mu))$ can be calculated as
\begin{equation}\label{eq7}
\phi(\mathbb{B}_t(\mu)) = \phi(\mu)+t
\end{equation}
for $t \geq -\phi(\mu)$.
\end{thm}
The following property was crucial to prove the semigroup property of $\mathbb{B}_t$ \cite{Bel2}. This will be also crucial in Theorem \ref{thm01} below.
\begin{prop}\label{prop5} Let $\mu \in  \mathcal{P}(\real)$ and $p, q$ be two real numbers such that $p \geq 1$ and $1-\frac{1}{p} < q$. Then
\begin{equation}
(\mu^{\boxplus p})^{\uplus q}= (\mu^{\uplus q'})^{\boxplus p'},
\end{equation}
where $p', q'$ are defined by $p' := pq/(1 - p + pq)$, $q' := 1 - p + pq$.
\end{prop}

\subsection{On free powers, Boolean powers and shifts}

The following results describe the behavior of the divisibility indicator under free powers, Boolean powers and shifts which give a better understanding of this indicator independently as an object independent of
the evolution $\mathbb{B}_{t}$. In particular, these results give a clear and quantitative descripition of why Bo\.zejko's conjecture is true.
First, we explicitly calculate the free divisibility indicator for free and Boolean time evolutions.
\begin{thm} \label{thm01} Let $\mu$ be a probability measure of $\mathcal{P}(\real)$. Then
 $\phi(\mu^{\uplus t}) = \frac{1}{t}\phi(\mu)$ for $t > 0$. Moreover, $\phi(\mu^{\boxplus t})-1 = \frac{1}{t}(\phi(\mu)-1)$ for $t > \max \{1- \phi(\mu),0\}$.
\end{thm}
\begin{proof}
If $\phi(\mu)=\infty$ this is trivial.
Suppose $\phi(\mu)=l$, then there is $\nu$ such that $\mathbb{B}_t(\nu)=\mu$ and then
\[
\begin{split}
\mu^{\uplus s}  &= \Big((\nu^{\boxplus (1+l)})^{\uplus \frac{s+l}{1+l}}\Big) ^{\uplus \frac{s}{s+l}} \\
&= \Big((\nu^{\uplus s})^{\boxplus \frac{s+l}{s}}\Big) ^{ \uplus \frac{s}{s+l}} \\
&= \Big((\nu^{\uplus s})^{\boxplus (1 +l/s)}\Big) ^{ \uplus \frac{1}{1+l/s}} \\
&= \mathbb{B}_{l/s}(\nu^{\uplus s}),
\end{split}
\]
where Proposition \ref{prop5} was applied in the second line with $p =1+l$, $q = \frac{s+l}{1+l}$. Therefore, $\phi(\mu^{\uplus s}) \geq l/s = \frac{\phi(\mu)}{s}$. Replacing $s$ by $1/s$ and $\mu$ by $\mu^{\uplus s}$, we see that $\phi(\mu) \geq s \phi(\mu^{\uplus s})$. Therefore, the conclusion follows for Boolean powers. From this result and (\ref{eq7}), we can prove that $\phi(\mu^{\boxplus (t+1)}) = \phi(\mathbb{B}_t(\mu)^{\uplus (1+t)}) = \frac{1}{1+t}\phi(\mathbb{B}_t(\mu)) = \frac{\phi(\mu)+t}{1+t}$ for $t > \max\{-\phi(\mu),-1 \}$.
\end{proof}

The first identity of Theorem \ref{thm01} leads to a new interpretation of $\phi(\mu)$ in terms of Boolean powers. Let us mention that this characterization is very useful for deriving the value of the free divisibility indicator, as we will see in the next section. This is because, in practice, it is much easier to calculate Boolean powers than free powers.
\begin{cor} \label{cor90}
$\phi(\mu) = \sup \{t \geq 0: \mu^{\uplus t} \text{~is~} \text{$\boxplus$-infinitely divisible}\}$.
\end{cor}
Bo\.zejko's conjecture follows immediately.
\begin{prop} \label{Boz}
If $\mu$ is  $\boxplus$-infinitely divisible, then so is $\mu^{\uplus t}$ for $0 \leq t \leq 1$.  Moreover,
\[
\phi_{\mu^{\uplus t}}(z) = K_{(\mu^{\boxplus (1-t)})^{\uplus t /(1-t)}}(z) =  \frac{t}{1-t}K_{\mu^{\boxplus (1-t)}}(z)
\]
for $0 < t < 1$. In terms of the Boolean Bercovici-Pata bijection $\Lambda_B$,

\begin{equation}\label{Bozeq}\Lambda_B \Big((\mu^{\boxplus (1-t)})^{\uplus t /(1-t)}\Big) = \mu^{\uplus t}.
\end{equation}
\end{prop}
\begin{proof}
Infinite divisibility is immediate from Theorem \ref{thm01}.
For $t \in (0,1)$,
\[
\begin{split}
\Lambda_B \Big((\mu^{\boxplus (1-t)})^{\uplus t /(1-t)}\Big) &= \mathbb{B}_1((\mu^{\boxplus (1-t)})^{\uplus t /(1-t)}) \\
&= \Big(\Big( (\mu^{\boxplus (1-t)})^{\uplus t /(1-t)} \Big)^{\boxplus 2} \Big)^{\uplus 1/2} \\ &= \Big(\Big( (\mu^{\uplus 2t})^{\boxplus 1/2} \Big)^{\boxplus 2} \Big)^{\uplus 1/2} \\
&= \mu^{\uplus t},
\end{split}
\]
where Proposition \ref{prop5} was applied in the third line.
\end{proof}
\begin{rem} Note that Equation (\ref{Bozeq}) also implies Bo\.zejko's conjecture and is independent of Theorem \ref{thm01}. However, Theorem \ref{thm01} gives a refinement of this fact.
\end{rem}
 
Free divisibility indicators are invariant under shifts of probability measures and also under the Boolean convolutions with delta measures. Let us first note the following.
\begin{lem}\label{lem51}Let $\mu$ be a probability measure on $\real$. Then the following are equivalent: \\
(1) $\mu$ is freely infinitely divisible; \\(2) $\mu \uplus \delta_a$ is freely infinitely divisible for any $a \in \real$.
\end{lem}
\begin{proof}
The following identity holds: 
$$
\phi_{\mu \uplus \delta_a}(z)=a+\phi_\mu(z+a). 
$$
The conclusion follows from the application of Theorem \ref{thmid}.   
\end{proof}
Now, Corollary \ref{cor90} combined with the above lemma enables us to prove the invariance of $\phi(\mu)$ under shifts and the Boolean convolution with $\delta_a$.

\begin{prop}\label{prop47} Let $\mu$ be a probability measure on $\real$ and $a \in \real$. Then $\phi(\mu) = \phi(\mu \uplus \delta_a) = \phi(\mu \boxplus \delta_a)$.
\end{prop}
\begin{proof}
For $t \geq 0$, we have $(\mu\uplus \delta_a)^{\uplus t} = \mu^{\uplus t} \uplus \delta_{at}$. Lemma \ref{lem51} then implies that $\mu^{\uplus t}$ is freely infinitely divisible if and only if $(\mu \uplus \delta_a)^{\uplus t}$ is so. Therefore, $\phi(\mu) = \phi(\mu \uplus \delta_a)$ from Corollary \ref{cor90}.

The second identity follows similarly from the identity
$$
(\mu \boxplus \delta_a)^{\uplus t} = (\mu^{\uplus t} \boxplus \delta_a) \uplus \delta_{(t-1)a}.
$$
\end{proof}

\subsection{On Jacobi parameters and free Meixner laws}
In this section, we give an upper bound for the free divisibility indicator in terms of Jacobi parameters. We start from the definition of Jacobi parameters~\cite{H-O}. For a probability measure $\mu$ with all finite moments, let us orthogonalize the sequence $(1,x,x^2,x^3,\cdots )$ in the Hilbert space $L^2(\real,\mu)$, following the Gram-Schmidt method. This procedure yields orthogonal polynomials $(P_0(x), P_1(x), P_2(x), \cdots)$ with $\text{deg}\, P_n(x) =n$. Multiplying constants, we take $P_n$ to be monic, i.e., the coefficient of $x^n$ is one. It is known that they satisfy a recurrence relation
\[
xP_n(x) = P_{n+1}(x) +\beta_nP_n(x) + \gamma_{n-1} P_{n-1}(x)
\]
for $n \geq 0$, under the convention that $P_{-1}(x)=0$. The coefficients $\beta_n$ and $\gamma_n$ are called Jacobi parameters and they satisfy $\beta_n \in \real$ and $\gamma_n \geq 0$. Indeed, it is known that $\gamma_0 \cdots \gamma_n=\int_{\real}|P_{n+1}(x)|^2\mu(dx)$ for $n \geq 0$. Moreover, the measure $\mu$ has a finite support of cardinality $N$ if and only if $\gamma_{N-1}=0$ and $\gamma_n > 0$ for $n = 0,\cdots, N-2$.
We write Jacobi parameters as
\[
J(\mu )=\left(
\begin{array}{ccccc}
\beta _{0}, & \beta _{1}, & \beta _{2}, & \beta _{3}, & ... \\
\gamma _0, & \gamma _1, & \gamma _2, & \gamma _3, & ...%
\end{array}%
\right).
\]
Continued fraction representation of $G_\mu$ can be expressed in terms of the Jacobi Parameters:
\[
\int_{\real}\frac{\mu(dx)}{z-x} = \dfrac{1}{z-\beta_0 -\dfrac{\gamma_0}{z-\beta_1-\dfrac{\gamma_1}{z- \beta_2 - \cdots}}}.
\]
This is useful to calculate $G_\mu$ from the Jacobi parameters.

In this section, we also consider Jacobi parameters $\beta_0,\beta_1, \gamma_0,\gamma_1$ for a probability measure $\mu$ with a finite fourth moment. We can define them just by orthogonalizing the sequence $(1,x,x^2)$ in $L^2(\real, \mu)$. The resulting orthogonal monic polynomials $(P_0(x), P_1(x), P_2(x))$ are expressed as
\[
P_0(x)=1,~~ P_1(x) = x-\beta_0,~~P_2(x)=(x-\beta_0)(x-\beta_1)-\gamma_0
\]
for some real numbers $\beta_0,\beta_1,\gamma_0$. Then $\gamma_0 =  \int_\real|P_1(x)|^2\mu(dx)$. $\beta_0$ is the mean of $\mu$ and $\gamma_0$ coincides with the variance of $\mu$. If $\gamma_0 =0$, we define $\gamma_1:= 0$, and if $\gamma_0\neq 0$, we define $\gamma_1 := \gamma_0^{-1}\int_\real|P_2(x)|^2\mu(dx)$.

\begin{prop}\label{prop38}
Let $\mu\neq\delta_a$ be a probability measure with a finite fourth moment. Let us consider the Jacobi parameters $\left(\gamma _{i},\beta _{i}\right)_{i=0,1}$ of $\mu$. Then $\phi (\mu )\leq \gamma _{1}/\gamma _{0}$.
\end{prop}
\begin{proof}
It was observed in \cite{BW2} that the Boolean power by $t$ is nothing else than multiplying both $\beta _{0}$ and $\gamma _{0}$ by $t$.
On the other hand, it was proved in \cite{Ml} that if a probability measure is $\boxplus $-infinitely divisible, then $\gamma _{0}\leq \gamma _{1}$ (this is still true under the assumption of finite fourth moment). Therefore, $\phi (\mu )\leq \gamma _{1}/\gamma _{0}$ from Theorem \ref{thm01}.
\end{proof}

\begin{cor}
Let $\mu\neq\delta_a$ be a probability measure with finite fourth moment. Then $\mu^{\uplus t} \notin ID(\boxplus)$ for $t$ large enough.
\end{cor}

The last proposition extends Theorem 3.7 in \cite{APA} which gives an upper bound for the divisibility indicator in terms of the Boolean kurtosis: $Kurt^\uplus(\mu)\geq\phi(\mu)$. Indeed, for a probability measure with mean zero and finite fourth moment, we have $Kurt^\uplus(\mu) = (\gamma _{1}+\beta_1^2)/\gamma _{0} \geq \gamma_1 /\gamma_0$. Therefore, Proposition \ref{prop38} gives a sharper estimate.
\begin{exa}The family of $q$-Gaussian distributions introduced by
Bo\.{z}ejko and Speicher in \cite{BS} (see also the paper \cite{BKS} of Bo\.{z}ejko, K\"{u}mmerer and
Speicher) interpolate between the normal ($q = 1$) and the semicircle ($q = 0$) laws. It is determined in terms of their orthogonal polynomials $H_{n,q}(x)$, called the $q$-Hermite polynomials, via the recurrence relation
\begin{equation*}
xH_{n,q}(x) = H_{n+1,q}(x) +\frac{1-q^n}{1-q} H_{n-1,q}(x).
\end{equation*}
It was proved in \cite{ABBL} that the $q$-Gaussian distributions are freely infinitely divisible for all $q\in[0,1]$. A direct application of the Proposition \ref{prop38} gives the estimate $\phi (\mu )\leq \gamma _{1}/\gamma _{0}=1+q$. In particular, this shows that the $q$-Gaussian distributions are not freely infinitely divisible when $q\in[-1,0)$.
\end{exa}
Notice from last example that the semicircle distribution satisfies the equality $\phi (\mu )= \gamma _{1}/\gamma _{0}$. In the next example we will see that the class of free Meixner distributions also satisfies this equality.
\begin{exa}
The free Meixner distributions $\mu _{\beta _{0},\gamma _{0},\beta_{1},\gamma _{1}}$ are probability measures with Jacobi parameters
\[
J(\mu )=\left(
\begin{array}{ccccc}
\beta _{0}, & \beta _{1}, & \beta _{1}, & \beta _{1}, & ... \\
\gamma _0, & \gamma _1, & \gamma _1, & \gamma _1, & ...%
\end{array}%
\right)
\]%
where $\beta _{0}, \beta_{1} \in {\mathbb{R}} $ and $\gamma _{0},\gamma_{1}\geq 0 $. More explicitly, $\mu _{0,1,b,1+c}$ is written as
\[
\frac{1}{2\pi }\cdot   \frac{\sqrt{[4(1+c)-(x-b)^{2}]_+}}{1+bx+cx^{2}}dx+ (\textup{0, 1 or 2 atoms),}
\]%
where $f(x)_+$ is defined by $\max\{f(x),0 \}$.
General free Meixner distributions are affine transformations of this case; see \cite{Ans}, \cite{BB}.

Let $\gamma=\gamma _{1}-\gamma _{0}$ and $\beta=\beta _{1}-\beta _{0}$. Then
\begin{equation}\label{eq100}
J(\mu ^{\boxplus t})=\left(
\begin{array}{ccccc}
\beta _{1}t, & \beta+\beta _{0}t, & \beta+\beta _{0}t, &\beta+\beta _{0}t, & ... \\
\gamma _{0}t, & \gamma+\gamma _{0}t, & \gamma+\gamma _{0}t, &\gamma+\gamma _{0}t, & ...%
\end{array}%
\right),
\end{equation}
as shown in \cite{A-M}.
On the other hand, from \cite{BW2}, we have
\begin{equation}\label{eq101}
J(\mu ^{\uplus s})=\left(
\begin{array}{ccccc}
\beta _{0}s, & \beta _{1}, & \beta _{1}, & \beta _{1}, & ... \\
\gamma _{0}s, & \gamma _1, & \gamma _1, & \gamma _1, & ...%
\end{array}%
\right).
\end{equation}

Now, it is easy to derive the free divisibility indicator of $\mu _{\beta _{0},\gamma _{0},\beta_{1},\gamma _{1}}$. Let us denote by $\gamma_i(\nu)$ Jacobi parameters of $\nu$ to distinguish a probability measure. We get $\gamma_1(\mu _{\beta _{0},\gamma _{0},\beta_{1},\gamma _1}^{\uplus t})-\gamma_0(\mu _{\beta _{0},\gamma _0, \beta_{1},\gamma _1}^{\uplus t}) = \gamma_1(\mu _{\beta _{0},\gamma _{0},\beta_{1},\gamma _1}) - t \gamma_0(\mu _{\beta _{0},\gamma _{0},\beta_{1},\gamma _1})$ using (\ref{eq101}). By the way, Saitoh and Yoshida \cite{S-Y} showed that a free Meixner distribution $\mu$ is $\boxplus$-infinitely divisible if and only if $\gamma_{1}(\mu)-\gamma_{0}(\mu) \geq 0$. Therefore, we conclude that
\[
\phi(\mu _{\beta _{0},\gamma _{0},\beta_{1},\gamma _{1}}) = \gamma_1(\mu _{\beta _{0},\gamma _{0},\beta_{1},\gamma _1})/\gamma_0(\mu _{\beta _{0},\gamma _{0},\beta_{1},\gamma _1})
\]
from Corollary \ref{cor90}.
%it is clear that $\phi (\mu _{\beta _{0},\gamma _{0,},\beta_{1},\gamma _1})$ does not depend on $\beta _{0}$ or $\beta_{1}$. Also, since $\mu _{0,1,0,1}$ is the semicircle distribution, $\phi (\mu _{0,1,0,1})=1$. Hence, combining Theorem \%ref{thm01} with Equations \ref{eq100} and \ref{eq101} it is readily seen that
%$\phi (\mu _{0,s,0,t})=t/s$, which implies that $\phi (\mu _{\beta _{0},\gamma %_{0,},\beta
%_{1},\gamma _{1,}})=\gamma _{1}/\gamma _{0}$.
In particular, for a so-called Kesten-McKay distribution $\mu _{t}$,
with absolutely continuous part
\[
\frac{1}{2\pi }\cdot \frac{\sqrt{4t-x^{2}}}{1-(1-t)x^{2}},
\]%
we get  $\phi (\mu _{t})=t$. This distribution was studied by
Kesten \cite{Kes} in connection to simple random walks on free groups. They appear in free probability theory as free additive powers of a Bernoulli distribution%, $(\frac{1}{2}\delta_{-1}+\frac{1}{2}\delta _{1})^{\boxplus n}$
. For relations between $d$-regular graphs and these distributions, see \cite{McK}.
\end{exa}

As shown in the above example, every free Meixner distribution achieves the upper bound of Proposition \ref{prop38}. More strongly, this characterizes free Meixner distributions.
\begin{thm} Let $\mu\neq\delta_a$ be a probability measure with a finite fourth moment and Jacobi parameters $\left(\gamma _{i},\beta _{i}\right)_{i=0,1}$. Then $\phi (\mu) = \gamma _{1}/\gamma _{0}$ if and only if $\mu$ is a free Meixner distribution.
\end{thm}
\begin{proof}
Let us use the notation $\gamma_i(\mu)$ to distinguish probability measures.
We assume that $\phi (\mu) = \gamma_{1}(\mu)/\gamma _{0}(\mu)$.
If $\gamma_0(\mu) = 0$, $\mu$ is a delta measure at a point, so that $\phi(\mu)=\infty$. We assume that the variance $\gamma_0(\mu)$ is nonzero.

First, let us suppose that $\gamma_1(\mu)=0$. From the paragraph previous to Proposition \ref{prop38}, the $L^2$-norm of $P_2(x)$ is zero, which implies that $\mu$ is supported on at most two points. In other words, $\mu$ is of the form $p\delta_a + (1-p)\delta_b$, i.e., $\mu$ is a Bernoulli law. This is a free Meixner law since its Jacobi parameters are given by
\[
J(\mu)=\left(
\begin{array}{ccccc}
\beta _{0}(\mu), & \beta _{1}(\mu), & \beta _{1}(\mu), & \beta _{1}(\mu), & ... \\
\gamma _{0}(\mu), & 0,& 0, & 0, & ...%
\end{array}%
\right),
\]
where $\beta_0(\mu) = pa + (1-p)b$, $\gamma_0(\mu)= p(1-p)(b-a)^2$ and $\beta_1(\mu)= pb + (1-p)a$. This calculation can be checked by using the continued fraction of the Stieltjes transform.

Next let us assume that $t:=\phi(\mu)=\gamma_1(\mu)/\gamma_0(\mu) >0$. Then there exists $\nu$ such that $\mathbb{B}_t(\nu)=\mu$. Using the relation $\phi(\mathbb{B}_t(\mu))=\phi(\mu)+t$, we conclude $\phi(\nu)=0$. Since $\gamma_0$ is equal to variance, $\gamma_0(\mu)=\gamma_0(\nu^{\boxplus (1+t)})/(1+t) = \gamma_0(\nu)$. From the relation (3.14) of \cite{Ml} and (3.9) of \cite{BW2},
\[
\gamma_1(\mu)
= \gamma_1(\nu^{\boxplus (1+t)}) = \gamma_1(\nu) +t\gamma_0(\nu).
\]
Divided by $\gamma_0(\mu)=\gamma_0(\nu)$, the above equality becomes
\[
\frac{\gamma_1(\mu)}{\gamma_0(\mu)} = \frac{\gamma_1(\nu)}{\gamma_0(\nu)} + t.
\]
Using the assumption $\phi(\mu)=\gamma_1(\mu)/\gamma_0(\mu)$, we have
\[
\phi(\mu)= t = \frac{\gamma_1(\mu)}{\gamma_0(\mu)} = \frac{\gamma_1(\nu)}{\gamma_0(\nu)} + t.
\]
Therefore, $\gamma_1(\nu)=0$, so that $\nu$ is a Bernoulli law. (\ref{eq100}) and (\ref{eq101}) imply that the set of the free Meixner laws is closed under the operation $\mathbb{B}_t$, so that $\mu=\mathbb{B}_t(\nu)$ is also a free Meixner law.
\end{proof}

\subsection{On free stable and Boolean stable laws}
We investigate free and Boolean stable distributions in terms of free divisibility indicators.
Free and Boolean stable laws were respectively introduced in \cite{Be-Vo} and \cite{S-W}, and their domains of attraction were studied in \cite{Be-Pa}.
Using Theorem \ref{thm01}, we show that the free divisibility indicator of a Boolean stable distribution takes only two possible values, $0$ or $\infty $. Similarly, in the free case we get two possibilities, $1$ or $\infty .$

Let $D_t$ denote the dilation operation defined by $(D_t \mu)(B):=\mu(t^{-1}B)$ for every Borel set $B$. Since, $\mathbb{B}_{t}(D_{s}(\mu))=D_{s}(\mathbb{B}_{t}(\mu ))$, for $t,s>0$, $\phi (\mu )$ does not change under dilations.

In this paper, a probability measure $\nu _{\alpha }$ is said to be $\uplus$-stable of index $\alpha $ if $\nu _{\alpha }\uplus \nu _{\alpha }=D_{1/2^{\alpha }}(\nu
_{\alpha })\uplus\delta_b$, for some $b\in\mathbb{R}$. If $b=0$ the probability measure $\nu _{\alpha }$ is said to be $\uplus$-strictly stable of index $
\alpha$. Now it is clear from Theorem \ref{thm01} and Proposition \ref{prop47} that a $\uplus$-stable law $\nu _{\alpha }$
satisfies
\[
\frac{1}{2}\phi (\nu _{\alpha })=\phi (\nu _{\alpha }\uplus \nu _{\alpha
})=\phi (D_{1/2^{\alpha }}(\nu _{\alpha }))=\phi (\nu _{\alpha })\text{ }
\]%
which is possible only if $\phi (\nu _{\alpha })=0$ or $\phi (\nu _{\alpha
})=\infty $, depending on whether $\mu $ is freely infinitely divisible or
not. Note that both $0$ and $\infty$ can be achieved since the Cauchy
distribution ($\alpha =1$) is freely infinitely divisible and the Bernoulli
distribution ($\alpha =2$) is not.

Analogously, a probability measure $\sigma _{\alpha }$ is said to be $\boxplus$-stable of index $%
\alpha $ if $\sigma _{\alpha }\boxplus \sigma _{\alpha }=D_{1/2^{\alpha
}}(\sigma _{\alpha })\boxplus\delta_b$, for some $b\in\mathbb{R}$. Also, $\sigma _{\alpha }$ is said to be $\boxplus$-strictly stable when $b=0$. As in the Boolean case, we see that
\[
\frac{1}{2}(1-\phi (\sigma _{\alpha }))=(1-\phi (\sigma _{\alpha })).
\]
This yields only two possibilities, either $\phi (\sigma _{\alpha })=1$ or $%
\phi (\sigma _{\alpha })=\infty$. Also in the free case, both
values $1$ and $\infty$ can be achieved: the semicircle law ($\alpha =2$)
satisfies $\phi (\sigma _{2})=1$~\cite{Bel2}; the Cauchy distribution $\sigma_1$, which is $\boxplus$-strictly stable of index $\alpha =1$, satisfies $\phi(\sigma _{1})=\infty$. Later we also consider free and Boolean $1/2$-stable laws.
\begin{rem}\label{rem95}
An upper bound can be estimated for the free divisibility indicator of the Gaussian distribution $N(0,1)$, i.e., the $2$-stable law in probability theory. Indeed, one can
show, using Theorem \ref{thm01} and numerical computations of cumulants\footnote{The authors would like to thank Professor Franz Lehner for his assistance to improve these numerical computations. }, that $\phi (N(0,1))<1.2$. The lower bound $\phi (N(0,1))\geq 1$,
which is much harder to prove, is implicit in Belinschi et al.~\cite{BBLR}.
\end{rem}
The free divisibility indicators of Cauchy and delta distributions are infinity since they are the fixed points of $\mathbb{B}_t$~\cite{Bel2}. The following theorem shows that there are other measures
with their free divisibility indicators equal to $\infty$, while they are not fixed points of $\mathbb{B}_t$.

\begin{thm}
Any $\uplus$-stable distribution $\nu$ of index $1/2$ is $\boxplus$-infinitely divisible. Moreover, $\phi(\nu)=\infty$.
\end{thm}
\begin{proof}
Thanks to Proposition \ref{prop47}, we may assume that $\nu$ is $\uplus$-strictly stable.
The reciprocal Cauchy transform of $\nu$ is given by~\cite{S-W}%
\[
F_{\nu }(z)=z+bz^{1/2},~~~~0\leq \arg b\leq \frac{\pi}{2}.
\]%
The Voiculescu transform then becomes
\[
\phi _{\nu }(z)=\frac{b^{2}}{2}-\sqrt{b^{2}z+b^{4}/4}.
\]%
$\phi _{\nu }(z)$ can be extended to a continuous function from $(\mathbb{C}^+\cup\mathbb{R}) \cup \{\infty \}$ to $\comp \cup \{\infty\}$. In addition, this mapping is homeomorphic. To understand the situation, it helps us to define $\psi:= J\circ \phi_\nu \circ J^{-1}$ as a function in $\overline{\disc}=\{|z| \leq 1 \}$, where $J:\comp^+ \to \disc$ is an analytic isomorphism. To prove $\nu$ is freely infinitely divisible, it is sufficient to prove that $\im \phi _{\nu }(x)$ takes non-positive values on $\mathbb{R}$. Indeed, let us suppose $\im \phi _{\nu }(x)\leq 0$ for $x\in \mathbb{R}$. This assumption is equivalent to $\psi(\partial \disc) \subset \disc^c = \{z \in \comp:|z| \geq 1\}$. From the homeomorphic property, $\psi(\partial \disc)$ is equal to $\partial \psi(\disc)$, and moreover, is a Jordan curve. Therefore,  $\psi(\partial \disc)$ divides $\comp$ into two simply connected open sets. One does not intersect with $\overline{\disc}$ and the other includes $\disc$. We can easily observe that $\lim_{y\to \infty}\im \phi_\nu(iy) = -\infty$, which implies that $\psi(\disc)$ coincides with the first one.

%\[
%\tau_{\nu}(dx):=\frac{-\im \phi _{\nu }(x)}{\pi (1+x^{2})}dx
%\]
%is a non-negative finite measure. Therefore, $-\widetilde{\phi}_{\nu }(z):=\int
%\frac{1+xz}{x-z}\tau _{\nu}(dx)$ is a Pick function, i.e., an analytic mapping from $\comp_+$ into $\comp_- \cup \real$.

%Since $\im \phi_\nu$ and $\im \widetilde{\phi}_\nu$ are harmonic in $\{(x,y) \in \real^2:y>0\}$ with the same boundary value,
%they coincide in the upper half-plane; this is a consequence of a basic theory of harmonic functions. The real part of an analytic map is determined by the imaginary part up to a constant; this follows from the Cauchy-Riemann equations. Now we can conclude that $\phi_\nu= \widetilde{\phi}_\nu + a$ for an $a \in \real$. Therefore, $-\phi_\nu$ is also a Pick function, which implies the free infinite divisibility.

Now, we prove that $\im \phi _{\nu }(x)\leq 0$ for any $x\in \mathbb{R}$. Let $c=b^{2}$ and suppose, without loss of generality,  that $c=e^{i\theta }$ $(0\leq \theta \leq \pi)$. Then
\[
\im \phi _{\nu }(x)=\frac{\sin \theta }{2}-\frac{1}{\sqrt{2}}\left[
\left(x^{2}+\frac{\cos \theta }{2}x+\frac{1}{16}\right)^{1/2}-x\cos \theta -\frac{\cos
2\theta }{4}\right] ^{1/2}.
\]%
Let us define
\[
f(x):=\left(x^{2}+\frac{\cos \theta }{2}x+\frac{1}{16}\right)^{1/2}-x\cos \theta -\frac{%
\cos 2\theta }{4}.
\]%
It is easy to prove that $f$ decreases in $(-\infty ,0)$ and increases in $%
(0,\infty )$ with its minimum value $f(0)=\frac{\sin^{2}\theta }{2}.$ Therefore, $%
\im \phi _{\nu }(x)$ takes the maximum value $\im\phi _{\nu }(0)=0$, and hence  $\im\phi _{\nu }(x)\leq 0$ for any $x\in
\mathbb{R}$ as we wanted.
The fact $\phi (\nu)=\infty$ now follows from the discussion prior to Remark \ref{rem95}.
\end{proof}

\begin{cor}
Any free $1/2$-stable distribution $\sigma$ satisfies $\phi (\sigma )=\infty$.
\end{cor}
\begin{proof}
 A $\boxplus$-stable law $\sigma$ of index $1/2$ is just $\Lambda_B(\nu)$ where $\nu$ is a $\uplus$-stable law of index $1/2$. Therefore, we see that $\phi (\sigma )=\phi (\nu )+1=\infty +1=\infty$.
\end{proof}

\section{Composition semigroups for multiplicative convolutions}\label{sec3}
In this section, we prove that many results on additive convolutions have counterparts for multiplicative convolutions.
A transformation $f_\mu(z) = \log(\eta_\mu(e^z))$ is useful to understand such results intuitively.  Indeed,  in terms of $f_\mu$, multiplicative convolutions can be characterized in a way analogous to the additive ones. For instance, the multiplicative free and Boolean convolutions can be characterized by
\[
f^{-1}_{\mu \boxtimes \nu} (z) = f^{-1}_\mu(z) + f^{-1}_\nu(z) -z,~~~f_{\mu \hutimes \nu}(z) = f_\mu(z) + f_\nu(z)-z.
\]
Therefore, $f_\mu(z)$, $f_\mu^{-1}(z)-z$ and $f_\mu(z)-z$ play the same roles as the reciprocal Cauchy transform, the Voiculescu transform and the energy function, respectively. From this observation, we can expect many results on multiplicative convolutions.

However, we cannot avoid the following problems.
\begin{itemize}
\item[(A)] On the unit circle, $\mu^{\hutimes t}$ for $t > 0$ and $\mu^{\boxtimes t}$ for $t > 1$ can be defined for any $\putimes$-infinitely divisible measure $\mu$, but these powers are not unique \cite{Bel3,Fra2}.
\item[(B)] On the positive real line,  Boolean powers of a generic probability measure can be defined only for a finite time \cite{Ber2}.
\end{itemize}
We discuss the above problems in this section.
%They sometimes do not matter, and in such cases, proofs can be omitted since they are the same as the additive case.

The essence of the problem (A) can be understood in terms of the universal covering of the Riemannian surface $\disc \backslash \{0\}$. We, however, do not use such a concept to avoid introducing many terminologies.

 \subsection{A composition semigroup on the unit circle}
Let $\id (\putimes; \tor)$ be the set of $\putimes$-infinitely divisible distributions on $\tor$.
Results in this subsection are often trivial for the normalized Haar measure $\omega$, so that we define the set $\id (\putimes; \tor)_0:= \id (\putimes; \tor)\backslash\{\omega \}$. This set can be written as \cite{Fra2}
$$
\id (\putimes; \tor)_0= \{\mu\in \mathcal{P}(\tor): \eta_\mu \text{~does not vanish in~}\disc \text{~and~}m_1(\mu)=\eta_\mu'(0) \neq 0 \},  
$$
on which a multiplicative free power can be defined \cite{Bel3}. Hence this set is important for multiplicative free convolution as well as for Boolean one. 

 On the unit circle $\tor$, the multiplicative Bercovici-Pata bijection from $\putimes$ to $\boxtimes$ was considered in a paper \cite{Wang1}.  We denote that map by $\Lambda_{MB}$ and then  it satisfies
\begin{equation}\label{eq03}
k_{\mu}(z)= \Sigma_{\Lambda_{MB}(\mu)}(z).
\end{equation}
The bijection $\Lambda_{MB}$ is a homeomorphism from $\id (\putimes; \tor)$ to the set of the $\boxtimes$-infinitely divisible distributions.

For $\mu \in \id (\putimes; \tor)_0$, both $\mu^{\boxtimes t}$ for $t > 1$ and $\mu^{\hutimes t}$ for $t >0$ can be defined, but they are not unique. For the Boolean case, this ambiguity is due to the rotational freedom \cite{Fra2}. Because of this ambiguity, Boolean powers do not work well in some situations \cite{Has4}. However, we can overcome this difficulty by introducing a countable family of free and Boolean powers. Let $u_\mu^{(0)}$ be the function satisfying $k_\mu(z)=e^{u_\mu^{(0)}(z)}$ with $-\pi < \im u_\mu^{(0)}(0) < \pi$; see Eq.\ (\ref{levymb}). If one needs to consider $\im u_\mu^{(0)}(0) = -\pi$ or $\pi$, one can approximate $u_\mu^{(0)}$ using a sequence of probability measures $\mu_n$ satisfying $\im u_{\mu_n}^{(0)}(0) \searrow  -\pi$ or $\im u_{\mu_n}^{(0)}(0) \nearrow  \pi$. We can define a family of convolution semigroups $\mu^{\hutimes_n t}$ ($n \in \mathbb{Z}$) from the relation
\[
k_{\mu^{\sutimes_n t}}(z)=e^{tu_\mu^{(n)}(z)},
\]
where $u_\mu^{(n)} = -2\pi ni + u_\mu ^{(0)}$.
%The minus sign is imposed since the Taylor expansion of $k_\mu$ at $0$ starts from the inverse of the first moment of $\mu$.

Also in the free case, the ambiguity comes from the rotational freedom. We define free powers in terms of subordination functions.
Let $\Phi_t^{(n)}(z):= z e^{(t-1)u_\mu^{(n)}(z)}$ for $\mu \in \id (\putimes; \tor)_0$ and $t >1$. An analytic map $\omega_t^{(n)}: \disc \to \disc$ exists for any $t >1$ satisfying the following properties (see Theorem 3.5 of \cite{Bel3}):
\begin{itemize}
\item[($\Omega 1$)] $\Phi_t^{(n)}(\omega_t^{(n)}(z)) = z$ for $z \in \disc$,
\item[($\Omega 2$)] $|\omega_t^{(n)}(z)| \leq |z|$ in $\disc$, 
\item[($\Omega 3$)] $\omega_t^{(n)}$ is univalent.
\end{itemize}
The condition ($\Omega 1$) implies the uniqueness of $\omega_t^{(n)}$. From ($\Omega2$) and Proposition \ref{prop0}(2), 
a probability measure $\mu^{\boxtimes_n t} \in \mathcal{P}(\tor)$ ($n \in \mathbb{Z}$) exists satisfying 
\[
\eta_{\mu^{\boxtimes_n t}}(z) = \eta_\mu(\omega_t^{(n)}(z)).
\]
Because of this formula, we call $\omega_t^{(n)}$ an \textit{$n$-th subordination function}. 
From the property ($\Omega3$), $\eta_{\mu^{\boxtimes_n t}}$ does not vanish in $\disc \backslash\{0\}$ and  from ($\Omega3$), $\eta_\mu'(0)\neq0$.  
Therefore, $\mu^{\boxtimes_n t}$ also belongs to $\id (\putimes; \tor)_0$. If $\mu$ is the normalized Haar measure $\omega$, then $\omega^{\hutimes_n t}$ and $\omega^{\boxtimes_n t}$ are simply defined by $\omega$ itself. 

We define an analogue of the semigroup $\mathbb{B}_t$, paying attention to the rotational freedom.
\begin{defi}\label{def41}
For a real number $\varphi$, we define $[\varphi]$ to be an integer determined as follows:  
\begin{itemize}
\item[(1)] $[\varphi] =0$ if $\varphi \in I_0:= (-\pi, \pi)$; 
\item[(2)] $[\varphi] = n$ if $\varphi \in I_n:= [n\pi, (n+2)\pi)$ and $n \geq 1$; 
\item[(3)] $[\varphi]= n$ if $\varphi \in I_n:= ((n-2)\pi, n\pi]$ and $n \leq -1$. 
\end{itemize}
In particular, $[-\varphi] = -[\varphi]$. 
\end{defi}
Using this, we define continuous families of Boolean and free powers:
\[
\mu^{\hutimes_\varphi t}:=\mu^{\hutimes_{[\varphi]}t},~~~~\mu^{\boxtimes_\varphi t}:=\mu^{\boxtimes_{[\varphi]}t}
\]
for $\varphi \in \real$.
\begin{defi}\label{def12}
A family of maps $\{\mathbb{M}_t^{(n)} \}_{t\geq 0}$ from $\id (\putimes; \tor)$ into itself is defined by
\begin{equation}
\mathbb{M}_t^{(n)}(\mu) = (\mu^{\boxtimes_{\arg m_1(\mu)} (t+1)})^{\hutimes_{(t+1)\arg m_1(\mu)} \frac{1}{t+1}},
\end{equation}
where $m_1(\mu)=\int_{\tor} \zeta~\mu(d\zeta)$ and $\arg m_1(\mu)$ is taken to satisfy $n = [\arg m_1(\mu)]$.
If $\arg m_1(\mu) =  n\pi$, we can define $\mathbb{M}_t^{(n)}(\mu)$ to be the limit $\lim_{r \searrow 0}\mathbb{M}_t^{(n)}(\mu \boxtimes \delta_{e^{ir\text{sign}(n)}})$, where $\text{sign}(x) = 1, x \geq 0$ and $\text{sign}(x) = -1, x < 0$.
\end{defi}
\begin{rem} The essence of the above definition is to make the function $m_1(\mathbb{M}_t^{(n)}(\mu))$ constant with respect to $t \in \real_+$. 
%For this purpose, we allow the argument of the first moment to take any number of $\real$, not only of $(-\pi, \pi)$. 
%This idea is important in Proposition \ref{prop131}.
\end{rem}

$\mathbb{M}_t^{(n)}$ is injective on $\id (\putimes; \tor)$ for $t \geq 0$, but not continuous on $\id (\putimes; \tor)$ wrt weak convergence for $t \notin \mathbb{N}$. This discontinuity has the same origin as the function $m_1(\mu)^t$. We take a branch such that $\mathbb{M}_t^{(n)}$ is continuous in the subset $\{\mu \in \id (\putimes; \tor): m_1(\mu) \in \overline{\disc}\backslash [-1,0] \}$.  $\mathbb{M}_1^{(n)}$ coincides with the Bercovici-Pata bijection $\Lambda_{MB}$ for any $n$, as in the additive case. Therefore, $\mathbb{M}_k^{(n)}$ is also continuous on $\id (\putimes; \tor)$ for any $k \in \nat$ and does not depend on $n$, since $\mathbb{M}_k^{(n)}$ is the iteration of $\Lambda_{MB}$ by $k$ times (see Theorem \ref{thm141}).

The following is a key to the semigroup property of $\mathbb{M}_t^{(n)}$.
\begin{prop}\label{prop131} Let $\mu \in \id (\putimes; \tor)_0$ and $\arg m_1(\mu) \in \real$ be an arbitrary argument of $m_1(\mu)$. \\
(1) Let $p, q$ be two real numbers such that $p \geq 1$ and $1-\frac{1}{p} < q$. Then we have
\begin{equation}
(\mu^{\boxtimes_{\arg m_1(\mu)} \,\, p})^{\hutimes_{p \arg m_1(\mu)} \,q}= (\mu^{\hutimes_{\arg m_1(\mu)}\, q'})^{\boxtimes_{q'\arg m_1(\mu)}\,\, p'},
\end{equation}
where  $p', q'$ are defined by $p' := pq/(1 - p + pq)$, $q' := 1 - p + pq$. \\
(2) $(\mu^{\boxtimes_{\arg m_1(\mu)} \,\,t})^{\boxtimes_{t \arg m_1(\mu)} \,\,s} = \mu^{\boxtimes_{\arg m_1(\mu)}\,\, ts}$ for $t,s \geq 1$. \\
(3) $(\mu^{\hutimes_{\arg m_1(\mu)}\,\, t})^{\hutimes_{t \arg m_1(\mu)} \,\,s} = \mu^{\hutimes_{\arg m_1(\mu)} \,\,ts}$ for $t,s \geq 0$.
\end{prop}
\begin{proof}
We use the notation $z^t_{\arg z}$ to distinguish branches. More precisely, $z^t_{\arg z}$ is defined to be $e^{it\arg z  + t\log |z|}$ for any argument of $z\neq 0$. We note that
\begin{equation}\label{eq567}
(zw)^t_{\arg z + \arg w} = z^t_{\arg z} w^t_{\arg w}
\end{equation}
for any $z, w \neq 0$.

We first prove the following fact:  for the $[\arg m_1(\mu)]$-th subordination function $\omega_t$ associated to a probability measure $\mu \in \id (\putimes; \tor)_0$, one gets
\begin{equation}\label{eq114}
\frac{\omega_t(z)}{z} = \Big(\frac{\eta_{\mu^{\boxtimes_{\arg m_1(\mu)} \,\,t}}(z)}{z}\Big)^{1-1/t}_{t\arg m_1(\mu)},~~~~ \frac{\eta_{\mu^{\boxtimes_{\arg m_1(\mu)}\,\, t}}(z)}{z}= \Big(\frac{\omega_t(z)}{z}\Big)^{\frac{t}{t-1}}_{(t-1)\arg m_1(\mu)}.
\end{equation}
This is proved as follows. For simplicity, let $n:=[\arg m_1(\mu)]$. From the relation $\Phi_t^{(n)}\circ \omega_t^{(n)}(z) =z$ and (\ref{eq567}), we have
\[
\begin{split}
1 &= \frac{\omega_t^{(n)}(z)}{z} \Big( \frac{\omega_t^{(n)}(z)}{z} \frac{z}{\eta_{\mu^{\boxtimes_n t}}(z)} \Big)^{t-1}_{-\arg m_1(\mu)} \\
 &= \frac{\omega_t^{(n)}(z)}{z} \Big(\frac{\omega_t^{(n)}(z)}{z}\Big)^{t-1}_{(t-1)\arg m_1(\mu)} \Big(\frac{z}{\eta_{\mu^{\boxtimes_n t}}(z)} \Big)^{t-1}_{-t\arg m_1(\mu)}  \\
  &= \Big(\frac{\omega_t^{(n)}(z)}{z}\Big)^{t}_{(t-1)\arg m_1(\mu)} \Big(\frac{z}{\eta_{\mu^{\boxtimes_n t}}(z)} \Big)^{t-1}_{-t\arg m_1(\mu)}, 
\end{split}
\]
from which the desired relations follow.

For simplicity, let us introduce the notations $\lambda:=(\mu^{\boxtimes_{\arg m_1(\mu)} \,\,p})^{\hutimes_{p\arg m_1(\mu)}\, q}$ and $\nu:= (\mu^{\hutimes_{\arg m_1(\mu)}\,\, q'})^{\boxtimes_{q'\arg m_1(\mu)}\,\, p'}$. If the $n$-th subordination function for $\mu$ is denoted simply by $\omega_t$, one obtains
\[
\frac{\eta_\lambda(z)}{z} = \Big( \frac{\eta_{\mu^{\boxtimes_{\arg m_1(\mu)} \,\,p}(z) }}{z}\Big)^q_{p \arg m_1(\mu)} = \Big( \frac{\omega_p(z)}{z}\Big)^{\frac{pq}{p-1}}_{(p-1)\arg m_1(\mu)},
\]
where (\ref{eq114}) was applied.
On the other hand, if the $[q' \arg m_1(\mu)]$-th subordination function for $\rho:= \mu^{\hutimes_{\arg m_1(\mu)}\,\, q'}$ is denoted by $\sigma_t$, then ($\Omega 1$) implies that $\frac{\eta_\rho (\sigma_{p'}(z))}{z}= \Big(\frac{\sigma_{p'}(z)}{z}\Big)^{\frac{1}{p'-1}}_{q'(p'-1)\arg m_1(\mu)}$. We note that $q'(p'-1) = p-1$, and therefore,
\[
\begin{split}
\frac{\eta_\nu(z)}{z}
&= \frac{\eta_\rho(\sigma_{p'} (z))}{z} = \frac{\eta_{\rho}(\sigma_{p'} (z))}{\sigma_{p'}(z)} \frac{\sigma_{p'}(z)}{z} \\
&= \Big( \frac{\sigma_{p'}(z)}{z}\Big)^{\frac{p'}{p'-1}}_{(p-1)\arg m_1(\mu)} = \Big( \frac{\sigma_{p'}(z)}{z}\Big)^{\frac{pq}{p-1}}_{(p-1)\arg m_1(\mu)}.
\end{split}
\]
The above calculations have reduced the problem to proving $\omega_p = \sigma_{p'}$. Let us prove this. The second identity of (\ref{eq114}), $\mu$ and $\omega_t$ replaced by $\rho = \mu^{\hutimes_{\arg m_1(\mu)}\, q'}$ and $\sigma_t$ respectively, leads to
\[
\frac{\sigma_{p'}(z)}{z} = \Big( \frac{\eta_\mu(\sigma_{p'}(z))}{\sigma_{p'}(z)}\Big)^{p-1}_{\arg m_1(\mu)}.
\]
This relation says that $\sigma_{p'}$ is exactly the right inverse of $\Phi_p(z):= z \big( \frac{z}{\eta_\mu(z)}\big)^{p-1}_{\arg m_1(\mu)}$. Therefore,  $\sigma_{p'} = \omega_p$ from the uniqueness.

(2) and (3) follow easily from analogous and simpler arguments.
\end{proof}

The semigroup property of $\mathbb{M}_t^{(n)}$ is immediate from Proposition \ref{prop131}.
\begin{thm}\label{thm141}
$\mathbb{M}_{t+s}^{(n)} = \mathbb{M}_t^{(n)} \circ \mathbb{M}_s^{(n)} \text{~on~} \id (\putimes; \tor),~~t,s \geq 0,~n\in \mathbb{Z}$.
\end{thm}
\begin{proof}
Let us assume that $n = [\arg m_1(\mu)]$.
Let us take $p = \frac{t+s+1}{t+1}$ and $q = \frac{s+1}{t+s+1}$ in Proposition \ref{prop131} (1) and replace $\mu$ by $\mu^{\boxtimes_{\arg m_1(\mu)}(t+1)}$. Then  Proposition \ref{prop131} (1) and (2) say that
\begin{equation}\label{eq56}
\begin{split}
&\left( (\mu^{\boxtimes_{\arg m_1(\mu)}(t+1)})^{\hutimes_{(t+1)\arg m_1(\mu)} \frac{1}{t+1}}\right)^{\boxtimes_{\arg m_1(\mu)} (s+1)} \\
&~~~~~~~~~~~=  \left((\mu^{\boxtimes_{\arg m_1(\mu)}(t+1)})^{\boxtimes_{(t+1)\arg m_1(\mu)} \frac{t+s+1}{t+1}}\right)^{\hutimes_{(t+s+1)\arg m_1(\mu)} \frac{s+1}{t+s+1}} \\
&~~~~~~~~~~~ =  (\mu^{\boxtimes_{\arg m_1(\mu)}(t+s+1)})^{\hutimes_{(t+s+1)\arg m_1(\mu)} \frac{s+1}{t+s+1}}.
\end{split}
\end{equation}
Therefore $\mathbb{M}_s^{(n)} \circ \mathbb{M}_t^{(n)}$ can be calculated as
\[
\begin{split}
\mathbb{M}_s^{(n)} \circ \mathbb{M}_t^{(n)}(\mu)
&= \left( \left((\mu^{\boxtimes_{\arg m_1(\mu)} (t+1)})^{\hutimes_{(t+1)\arg m_1(\mu)} \frac{1}{t+1}}\right)^{\boxtimes_{\arg m_1(\mu)} (s+1)}\right) ^{\hutimes_{(s+1)\arg m_1(\mu)} \frac{1}{s+1}} \\
&=  (\mu^{\boxtimes_{\arg m_1(\mu)} (t+s+1)})^{\hutimes_{(t+s+1)\arg m_1(\mu)} \frac{1}{t+s+1}} \\
&= \mathbb{M}_{t+s}^{(n)}(\mu),
\end{split}
\]
where Proposition \ref{prop131} (3) and (\ref{eq56}) were applied in the second line.
\end{proof}
%Unfortunately, $\mathbb{M}_t^{(n)}$ does not seem to be a homomorphism with respect to $\boxplus$.

An analogue of the free additive divisibility indicator can be defined for the multiplicative case as follows.
\begin{defi}
We define a multiplicative free divisibility indicator $\theta^{(n)}(\mu)$ to be
\[
\theta^{(n)}(\mu)= \sup \{t \geq 0: \mu \in \mathbb{M}_t^{(n)}(\id(\putimes;\tor)) \}
\]
for $\mu \in \id(\putimes;\tor)$ and write $\theta(\mu):=\theta^{(0)}(\mu).$
\end{defi}
$\mathbb{M}_t^{(n)}(\mu)$ is well defined for $t \geq - \theta^{(n)}(\mu)$ for the same reason as the additive case. 
The quantity $\theta^{(n)}$ in fact is independent of $n$ as we state in the following proposition for further reference. 
\begin{prop}\label{di}
For all $n\in\mathbb{Z}$ and $\mu \in \id (\putimes; \tor)$, we have $\theta^{(n)}(\mu)=\theta(\mu)$.
\end{prop}
The proof  is passed to Appendix for fluent reading, because it contains complicated notation. 

The free convolution power $\mu^{\boxtimes_n \, t}$ was introduced for $\mu \in \id(\putimes; \tor)$ and for $t \geq 1$ in terms of subordination functions. As it is called a free power, it is also characterized by 
\begin{equation}\label{sigma}
\Sigma_{\mu^{\boxtimes_n  t}}(z) = \Sigma_\mu(z)^t. 
\end{equation}
The branch of $\Sigma_\mu(z)^t$ depends on $n$. Using the expansion $\Sigma_\mu(z)= \frac{1}{m_1(\mu)}(1+O(z))$, we define $\Sigma_\mu(z)^t$ to be $e^{-it \arg m_1(\mu) - t \log|m_1(\mu)|}(1+O(z))^t$, where $\arg m_1(\mu) \in I_n$; see Definition \ref{def41}. We do not indicate a specific $n$ in $\Sigma_\mu(z)^t$, but that will be always clear from the context. Even for $t <1$, if a measure $\mu^{\boxtimes_n  t} \in \id(\putimes; \tor)$ exists  satisfying Eq.\ (\ref{sigma}), let us call it a $t$-th free power.

Now we prove counterparts of Theorems \ref{thm001} and \ref{thm01}.
\begin{thm}\label{thm87} We consider a probability measure $\mu \in \id (\putimes; \tor)$. \\
(1) $\mu^{\boxtimes_n\, t}$ exists for any $n \in \mathbb{Z}$ and $t \geq \max\{1-\theta(\mu),0 \}$. \\
(2) $\mu$ is $\boxtimes$-infinitely divisible if and only if $\theta(\mu) \geq 1$. \\
(3) $\theta(\mathbb{M}_t^{(n)}(\mu)) = \theta(\mu)+t$ for any $n \in \mathbb{Z}$ and $t \geq -\theta(\mu)$. \\
(4) $\theta(\mu^{\hutimes_n \,t}) = \frac{1}{t}\theta(\mu)$ for any $n \in \mathbb{Z}$ and $t > 0$. \\
(5) $\theta(\mu^{\boxtimes_n\, t})-1 = \frac{1}{t}(\theta(\mu)-1)$ for any $n \in \mathbb{Z}$ and $t > \max \{1- \theta(\mu), 0 \}$.
\end{thm}
%\begin{rem}
%Since $M_k^{(n)}$ does not depend on $n$ for any positive integer $k$, the statement (2) is also equivalent to $\theta^{(n)}(\mu) \geq 1$ for any fixed $n$.
%\end{rem}
\begin{proof} All the proofs are similar to the additive case with slight modification. The reader is referred to Section 5 of \cite{Bel2} and Theorem \ref{thm01} of this paper. For instance, (4) can be proved as follows. Let us take $\arg m_1(\mu)$ such that $n = [\arg m_1(\mu)]$ and suppose $\theta^{(n)}(\mu)=t$. By definition, there exists $\nu$ such that $\mathbb{M}_t^{(n)}(\nu)=\mu$. We note that $\arg m_1(\mu) = \arg m_1(\nu)$. Then
\[
\begin{split}
\mu^{\hutimes_{\arg m_1(\mu)} s}  &= \Big((\nu^{\boxtimes_{\arg m_1(\mu)} (1+t)})^{\hutimes_{(t+1)\arg m_1(\mu)} \frac{s+t}{1+t}}\Big) ^{\hutimes_{(t+s)\arg m_1(\mu)} \frac{s}{s+t}} \\
&= \Big((\nu^{\hutimes_{\arg m_1(\mu)} s})^{\boxtimes_{s \arg m_1(\mu)} \frac{s+t}{s}}\Big) ^{ \hutimes_{(t+s)\arg m_1(\mu)} \frac{s}{s+t}} \\
&= \Big((\nu^{\hutimes_{\arg m_1(\mu)} s})^{\boxtimes_{s \arg m_1(\mu)} (1 +t/s)}\Big) ^{\hutimes_{(t+s)\arg m_1(\mu)} \frac{1}{1+t/s}} \\
&= \mathbb{M}_{t/s}^{([s\arg m_1(\mu)])}(\nu^{\hutimes_{\arg m_1(\mu)} s}),
\end{split}
\]
where we used Proposition \ref{prop131} (1) with $p =1+t$, $q = \frac{s+t}{1+t}$. Therefore, $\theta^{([s \arg m_1(\mu)])}(\mu^{\hutimes_n s}) \geq t/s = \frac{\theta^{(n)}(\mu)}{s}$. As in Theorem \ref{thm001}, let us replace $s$ by $1/s$ and $\mu$ by $\mu^{\hutimes_{\arg m_1(\mu)} s}$. Using Proposition \ref{prop131}(2), we obtain the converse inequality.
\end{proof}
We can immediately characterize the free divisibility indicator with Boolean multiplicative powers on the unit circle.
\begin{cor}
$\theta(\mu) = \sup \{t \geq 0: \mu^{\hutimes_n t} \text{~is~} \text{$\boxtimes$-infinitely divisible}\}$ for $\mu \in \id (\putimes; \tor)$.
\end{cor}
An analogue of Bo\.zejko's conjecture for multiplicative convolutions is also the case on the unit circle.
\begin{prop} If $\mu \in \id (\putimes; \tor)$ is $\boxtimes$-infinitely divisible, then so is $\mu^{\hutimes_n t}$ for $0 \leq t \leq 1$ and $n \in \mathbb{Z}$. Moreover, $\mu^{\hutimes_{n} t} = \Lambda_{MB}((\mu^{\boxtimes_{n} (1-t)})^{\hutimes_{(1-t)\arg m_1(\mu)} t/(1-t)})$ for $0 < t < 1$, where $n=[\arg m_1(\mu)]$.
\end{prop}
\begin{proof}
The proof is similar to Proposition \ref{Boz}.
\end{proof}

\begin{exa} In the literature, known examples are not many whose free power and Boolean power can be explicitly computed. Finding more examples may be a fruitful question in future. Here is shown one simple example.   
For $a \geq 0$ and $b \in \real$, let $\mu$ be a probability measure on $\tor$ defined by
\[
\mu(d\theta) = \frac{1}{2\pi}\frac{1 - e^{-2a}}{1 + e^{-2a}-2e^{-a}\cos(\theta - b)}d\theta, ~~0 \leq \theta < 2\pi.
\]
This is an analogue of the Cauchy distribution on $\real$ since $\mu$ is identical to the density of the Poisson kernel. Let $c:=e^{-a+ib}$, then $\eta_\mu(z) = cz$ and
\[
\Sigma_\mu(z) = k_{\mu}(z) = c^{-1}= e^{-ib}\exp\left(a\int_{\tor}\frac{1+\zeta z}{1-\zeta z}\omega(d\zeta)\right),
\]
where $\omega(d\theta)$ is the normalized Haar measure.
Therefore, $\mathbb{M}^{(n)}_t(\mu) = \mu$ for any $n \in \mathbb{Z}$ and any $t \geq 0$.  The free divisibility indicator $\theta(\mu)$ is equal to $\infty$.
\end{exa}

\subsection{A composition semigroup on the positive real line}
 From Proposition \ref{prop0}(i), a logarithm $\log k_\mu(z) = \log (z/\eta_\mu(z))$ can be defined in $\comp \backslash \real_+$ with values in $\comp$ for $\mu \neq \delta_0 \in \mathcal{P}(\real_+)$. The function $\log k_\mu(z)$ maps $\comp^+$ to $\comp^- \cup \real$, and therefore, it has the Pick-Nevanlinna representation
\begin{equation}\label{eq01}
\log k_\mu(z) = -a_\mu z + b_\mu + \int_{0}^\infty \frac{1+xz}{z-x}\tau_\mu(dx)
\end{equation}
for $a_\mu \geq 0$, $b_\mu \in \real$ and a non-negative finite measure $\tau_\mu$ on $\real_+$. This is, in a sense, a L\'{e}vy-Khintchine formula for the multiplicative Boolean convolution on $\real_+$. To understand a Bercovici-Pata bijection, we have to know when a function $K(z) = -a z + b + \int_{0}^\infty \frac{1+xz}{z-x}\tau(dx)$ can be written as $\log k_\mu(z)$ for a probability measure $\mu$ on $\real_+$. For instance, Proposition \ref{prop0} implies that $\im(\log k_\mu(z)) \in (- \pi + \arg z, 0]$ in $\comp^+$. In particular, $-\pi \leq \im (\log k_\mu) \leq 0$. Therefore, $a =0$. Moreover, $\tau$ cannot contain the singular part: if the singular part were non-zero, a point $x_0 \geq 0$ would exist such that $\im K(x_0+i0) = \infty$. These conditions, however, are far from a complete characterization.

In spite of the above, we can still construct an injective mapping $\Lambda_{MB}$ from $\mathcal{P}(\real_+)$ to the set of $\boxtimes$-infinitely divisible distributions with the relation
\begin{equation}\label{eq04}
k_{\mu}(z)= \Sigma_{\Lambda_{MB}(\mu)}(z).
\end{equation}
Let us call this map $\Lambda_{MB}$ a Bercovici-Pata map from $\utimes$ to $\boxtimes$ for probability measures on the positive real line.
As explained above, $a_{\Lambda_{MB}(\mu)} =0$ and $\tau_{\Lambda_{MB}(\mu)}$ is absolutely continuous with respect to the Lebesgue measure, where $a_\nu$ and $\tau_\nu$ have been defined in  (\ref{eq76}). Therefore, $\Lambda_{MB}$ is not surjective.

Now we define an analogue of the semigroup $\mathbb{B}_t$ for the multiplicative convolutions.
\begin{defi}\label{def122}
A family of maps $\{\mathbb{M}_t \}_{t\geq 0}$ from $\mathcal{P}(\real_+)$ into itself is defined by
\begin{equation}
\mathbb{M}_t(\mu) = (\mu^{\boxtimes (t+1)})^{\hutimes \frac{1}{t+1}}.
\end{equation}
For $\delta_0$, $\mathbb{M}_t(\delta_0)$ is defined to be just $\delta_0$. 
\end{defi}
As proved in \cite{Ber2}, $\mu^{\hutimes t} \in \mathcal{P}(\real_+)$ is defined for any probability measure $\mu \in \mathcal{P}(\real_+)$ and $0 \leq t \leq 1$. Therefore, $\mathbb{M}_t$ is well defined. The map $\mathbb{M}_t$ on the positive real line is simpler than on the unit circle, since a Boolean power is unique if exists.

The following result is essentially the same as the additive case,   except for the restriction $q \leq 1$.
\begin{prop}\label{prop13}
Let $p, q$ be two real numbers such that $p \geq 1$ and $1-\frac{1}{p} < q \leq 1$. We have
\begin{equation}
(\mu^{\boxtimes p})^{\hutimes q}= (\mu^{\hutimes q'})^{\boxtimes p'},
\end{equation}
where $\mu \in \mathcal{P}(\real_+)$ and $p', q'$ are defined by $p' := pq/(1 - p + pq)$, $q' := 1 - p + pq$.
All the convolution powers are well defined under the above assumptions.
\end{prop}
The proof is easier than that of Proposition \ref{prop131}; we do not have to pay attention to branches of analytic mappings.
The semigroup property holds also in this case.
\begin{thm}\label{thm142}
$\mathbb{M}_{t+s} = \mathbb{M}_t \circ \mathbb{M}_s \text{~on~} \mathcal{P}(\real_+),~~t,s \geq 0$.
\end{thm}

We can also define a free divisibility indicator:
\[
\theta(\mu)= \sup \{t \geq 0: \mu \in \mathbb{M}_t(\mathcal{P}(\real_+)) \}
\]
for $\mu \in \mathcal{P}(\real_+)$.
Some results of Theorem \ref{thm87} have no counterparts for probability measures on the positive real line. This is because the Bercovici-Pata map is not surjective and a Boolean power cannot be defined for a large time. We can, however, still prove the following.
\begin{prop} Let $\mu$ be a probability measure on $\mathcal{P}(\real_+)$. Then \\
(1) $\mu^{\boxtimes t}$ exists for $t \geq \max\{1-\theta(\mu),0 \}$, \\
(2) $\mu$ is $\boxtimes$-infinitely divisible if $\theta(\mu) \geq 1$, \\
(3) $\theta(\mathbb{M}_t(\mu)) = \theta(\mu)+t$ for $t \geq -\theta(\mu)$.
\end{prop}
\begin{proof}
All the proofs are similar to the previous cases (see Theorem \ref{thm87} of this paper and Section 5 of \cite{Bel2}). A remark on (2) is that the $\boxtimes$-infinite divisibility of $\mu \in \mathcal{P}(\real_+)$ does not imply $\theta(\mu) \geq 1$ since the Bercovici-Pata map is not surjective, see Example \ref{booleanstable}.
\end{proof}

\begin{exa}\label{booleanstable}
Let $\mathbf{b}_\alpha$ be a positive $\uplus$-strictly stable law with index $0<\alpha\leq1$ which is characterized by $\eta_  {\mathbf{b}_\alpha}(z)=-(-z)^\alpha$. 
We also have $\Sigma_{\mathbf{b}_\alpha}=(-z)^{\frac{1-\alpha}{\alpha}}$. So
$\Sigma_{(\mathbf{b}_\alpha)^{\boxtimes t}}=(-z)^{\frac{1-\alpha}{\alpha}t}$
and $$\eta_{(\mathbf{b}_\alpha)^{\boxtimes t}}=-(-z)^{\frac{\alpha}{(1-\alpha)t+\alpha}}.$$ This means, on one hand,  that $(\mathbf{b}_\alpha)^{\boxtimes t} = \mathbf{b}_{\frac{\alpha}{(1-\alpha)t+\alpha}}$ and hence $\mathbf{b}_\alpha$ is $\boxtimes$-infinitely divisible. 

On the other hand, 
$$\eta_{\left((\mathbf{b}_\alpha)^{\boxtimes t}\right)^{\sutimes 1/t}}=-(-z)^{\frac{2\alpha-1+(1-\alpha)t}{(1-\alpha)t+\alpha}}.$$
So, for $\beta={\frac{2\alpha-1+(1-\alpha)t}{(1-\alpha)t+\alpha}}$ or equivalently $\alpha = \frac{\beta - (1-\beta)t}{1- (1-\beta)t}$, we have $\mathbb{M}_t(\mathbf{b}_\alpha)=\mathbf{b}_{\beta}$. This means, given $\beta \in (0,1]$, $\mathbf{b}_{\beta}\in\mathbb{M}_t(\mathcal{P}(\real_+))$ if and only if $0<t<\frac{\beta}{1-\beta}$ and then 
$$
\theta(\mathbf{b}_\alpha)=\frac{\alpha}{1-\alpha}.
$$ 
To summarize, if $\alpha<1/2 $,  $\theta(\mathbf{b}_\alpha)<1$ but
$\mathbf{b}_\alpha$ is $\boxtimes$-infinitely divisible, showing that $\boxtimes$-infinite divisibility of $\mu$ does not imply that $\theta(\mu)\geq1.$
\end{exa}

\section{Commutation relations between Boolean and free convolution powers}\label{sec4}
We have seen that convolution powers for $\boxtimes$ and $\putimes$ on $\real_+$ satisfy the same commutation relation as the additive case (see Propositions \ref{prop5}, \ref{prop13}). Moreover, there are other commutation relations involving free and Boolean powers. We will prove such relations in this section. These are useful to construct new examples of $\boxtimes$-infinitely divisible distributions.

\begin{prop}\label{comm1}
The following commutation relations hold for $\mu \in \mathcal{P}(\real_+)$. \\
(1) $(\mu^{\boxplus t})^{\boxtimes s} = D_{t^{s-1}}(\mu^{\boxtimes s})^{\boxplus t}$ for $t\geq 1$ and $s \geq 1$. \\% such that $\mu^{\boxtimes s} \in \mathcal{P}(\real_+)$ exists.  \\
(2) $(\mu^{\uplus t})^{\boxtimes s} = D_{t^{s-1}}(\mu^{\boxtimes s})^{\uplus t}$ for $t \geq 0$ and $s \geq 1$. \\
% such that $\mu^{\boxtimes s} \in \mathcal{P}(\real_+)$ exists.  \\
(3) $(\mu^{\uplus t})^{\hutimes s} = (\mu^{\hutimes s})^{\uplus t^s}$ for $t \geq 0$ and $s \leq 1$. % such that $\mu^{\hutimes s} \in \mathcal{P}(\real_+)$ exists.
\end{prop}
\begin{proof} We note that $\mu^{\boxplus t}$ is supported on $\real_+$ for $t \geq 1$ (see discussions in Subsection 2.4 of \cite{Bel2}).  (1) and (2) were essentially proved in Proposition 3.5 of \cite{Bel2}.

(3) We recall the relations $\eta_{\mu^{\uplus t}}(z) = t \eta_\mu(z)$ and $\eta_{\mu^{\hutimes s}}(z)/z = (\eta_\mu(z)/z)^s$. For the left hand side, we have
\[
\frac{\eta_{(\mu^{\uplus t})^{\sutimes s}}(z)}{z} = \left(\frac{\eta_{\mu^{\uplus t}}(z)}{z}\right)^s = t^s \left(\frac{\eta_\mu(z)}{z}\right)^s,
\]
and for the right hand side,
\[
\frac{\eta_{(\mu^{\sutimes s})^{\uplus t^s}}(z)}{z} = t^s\frac{\eta_{\mu^{\sutimes s}}(z)}{z} = t^s \left(\frac{\eta_\mu(z)}{z}\right)^s.
\]
Therefore, they coincide.
\end{proof}
\begin{rem}
(i) The parameter $t$ in (1) may not be extended to $t \geq 0$ even if $\mu$ is $\boxplus$-infinitely divisible. This is because $\text{supp~} \mu \subset [0,\infty)$ does not imply $\text{supp~} \mu^{\boxplus t} \subset [0,\infty)$ for every $t \geq 0$. We will discuss this problem in another paper \cite{AHS}.  \\
(ii) In Propositions \ref{prop5}, \ref{prop13} and \ref{comm1}, we have derived five commutation relations among $\boxplus, \boxtimes, \uplus$ and $\putimes$. The only missing relation is for the pair $\boxplus$ and $\putimes$. The question of if there is an algebraic relation between these two convolutions is an open problem.
\end{rem}

The following result is immediate.
\begin{cor}\label{thm05}
If $\mu \in \mathcal{P}(\real_+)$ is $\boxtimes$-infinitely divisible, then so are $\mu^{\boxplus t}$ for $t \geq 1$ and $\mu^{\uplus s}$ for $s \geq 0$.
\end{cor}

Moreover, from (2) and (3) of Proposition \ref{comm1} we see that the multiplicative divisibility indicator does not change under the action of Boolean additive powers.
\begin{cor}
If $\mu \in \mathcal{P}(\real_+)$ then $\theta(\mu^{\uplus s})=\theta(\mu)$ for any $s>0$.
\end{cor}

It is well known that the free Poisson distribution $\pi$, characterized by $\phi_\pi(z) = \frac{z}{z-1}$, is both $\boxplus$ and $\boxtimes $-infinitely divisible. The following example generalizes this fact for different powers of $\pi$.
\begin{exa}
Let $\pi_{t,s,r} :=(({\pi }^{\boxplus t} )^{\boxtimes s} )^{\uplus r}$ for $r,s,t \geq 0$. It is clear from Corollary \ref{thm05} that $\pi_{t,s,r} $ is $\boxtimes $-infinitely divisible for $r,s \geq 0$ and $t \geq 1$.  Moreover, since  ${\pi }^{\boxplus t}$ is supported on $\real_+$ for every $t>0$, combining Propositions~\ref{Boz} and \ref{comm1}, we see that $\pi_{t,s,r} $ is $\boxplus$-infinitely divisible for $r \leq 1$, $s \geq 1$ and $t \geq 0$.
In particular, if $r \leq 1$ and $s,t \geq 1$, $\pi_{t,s,r}$ is infinitely divisible with respect to both $\boxplus$ and $\boxtimes$.

On the other hand, ${\pi }^{\boxplus t}$ is not $\boxtimes$-infinitely divisible for $t<1$ as shown by P\'{e}rez-Abreu and Sakuma; see Proposition 10 of \cite{PA-S}. This shows that we cannot extend Corollary~\ref{thm05} to $t<1$, even if ${\mu }^{\boxplus t}$ exists.
\end{exa}
\appendix

\section{Appendix} 
\subsection{Fixed points of $\mathbb{B}_t$}\label{sec5}
As we have shown, some measures have free divisibility indicators infinity as well as Cauchy distributions. So one may ask if this is because of some fixed point property. 
In this section we determine all the fixed points of the Boolean-to-free Bercovici-Pata bijection and more generally of $\mathbb{B}_t$ for each $t >0$. The key is the following functional equation for analytic maps.
\begin{lem}\label{lema1}
Let $F: \comp^+ \rightarrow \mathbb{C}^+$ be an analytic map such that
\begin{equation}\label{mastereq}
 F(z)-z=F(F(z))-F(z),~~z\in \mathbb{C}^+. 
\end{equation}
Then $F(z)=z+c$ for some $c \in \mathbb{C}^+ \cup \real$.
\end{lem}
\begin{proof} 
If $F$ is identity, (\ref{mastereq}) is trivially satisfied. So we may assume that $F$ is non identical. 
Note that $F$ is injective. Indeed, if $F(z)=F(w)$ then $F(F(z))=F(F(w))$ and then from Equation (\ref{mastereq}) $z=2F(z)-F(F(z))=2F(w)-F(F(w))=w$.

Now take $z_0\in \comp^+$ so that $F(z_0) \neq z_0$ and define $c:=F(z_0)-z_0$. Moreover, suppose that  $F$ is not identically equal to $z+c$. Let $D$ be a bounded domain such that $\overline{D}\subset\mathbb{C}^+$. Then, by the Identity Theorem, $F(z)=z+c$ for at most a finite number of points inside $D$. Hence, there exists a radius $r$, such that the ball $B_r(z_0)$ satisfies that $F(z)-z\neq c$ for all $z\in \overline{B_r(z_0)}\setminus \{z_0\}$. Moreover we may assume that $F(z_0) \notin \overline{B_r(z_0)}$  since $F(z_0)\neq z_0$. Since $F$ is injective, $F^{\circ n+1}(z_0)\notin F^{\circ n}(\overline{B_r(z_0)})$ for any $n\geq 1$.  Let us consider the curve $C=\partial B_r(z_0)$. Since $C$ is compact, there exists $t>0$ such that $|F(z)-z-c|>t$ for all $z\in C$. 

Take an arbitrary $z\in C$. If we write $F(z)-z=d$, then $|d-c|>t$. From the iterative use of (\ref{mastereq}), we have $F^{\circ n}(z)=z+nd$ and $F^{\circ n}(z_0)=z_0+nc$ and then we see that $|F^{\circ n}(z)-F^{\circ n}(z_0)|>tn-|z_0|-|z|>tn-2|z_0|-r$ which tends to $\infty$ as $n\to\infty$. Thus for all  $R>0$, there is $N$ large enough such $|F^{\circ n}(z)-F^{\circ n}(z_0)|>R$ for all $n>N$ and all $z\in C$. 

Finally, for each $n>0$, $F^{\circ n}(C)$ is a simple curve surrounding $F^{\circ n}(z_0)$. Hence, by the considerations above, for $n$ large enough $F^{\circ n}(C)$ encloses $B_{2|c|}(F^{\circ n}(z_0))$. In particular, since $F^{\circ n+1}(z_0)-F^{\circ n}(z_0)=c$, the curve $F^{\circ n}(C)$ must surround $F^{\circ n+1}(z_0)$, contradicting the fact $F^{\circ n+1}(z_0)\notin F^{\circ n}(\overline{B_r(z_0)})$.
%Note to Hasebe: we do not need the surrounding argument. To show  $F^{\circ n+1}(z_0)\in F^{\circ n}(\overline{B_r(z_0)})$ with is enough to notice that $B_{2|a+bi|}(F^{\circ n}(z_0))\in F^{\circ n}(\overline{B_r(z_0)})$ since the boundary  $F^{\circ n}(C)$ is far from $F^{\circ n}(z_0)$.
\end{proof}

\begin{thm}
Let $t>0$ be real and $\mu$ be a fixed point of $\mathbb{B}_t$, i.e. $\mathbb{B}_t(\mu)=\mu$. Then $\mu$ is a point measure or a Cauchy distribution $\gamma_{a,b}$ with density
$$
\gamma_{a,b}(x)=\frac{b}{\pi[(x-a)^2+b^2]},~~~ x\in\mathbb{R}
$$ 
for some $a \in \mathbb{R}$, $b >0$. 
\end{thm}

\begin{proof}
%Suppose $\mu$ is not a point measure, then $\text{Im}\, F(z)>\text{Im}\, z$ for any $z \in \mathbb{C}^+$. 
From the basic properties of free and boolean convolutions,  
\begin{equation} \label{eq3}
F_{\mu^{\boxplus (t+1)}}^{-1}(z)=(t+1)F_{\mu}^{-1}(z)-tz 
\end{equation}
and
\begin{equation} \label{eq4}
F_{\mu^{\uplus (t+1)}}(z)=(t+1)F_{\mu}(z)-tz. 
\end{equation}
Recall that $\mathbb{B}_t(\mu)=\mu$ is  equivalent to $\mu^{\boxplus (t+1)}=\mu^{\uplus (t+1)}$. Plugging (\ref{eq4}) into (\ref{eq3}) we have 
$$
z=F_{\mu^{\boxplus (t+1)}}^{-1}(F_{\mu^{\uplus (t+1)}}(z))=
(t+1)F_{\mu}^{-1}((t+1)F_{\mu}(z)-tz)-t\left((t+1)F_{\mu}(z)-tz\right),
$$
from which
 $$
 F_{\mu}^{-1}((t+1)F_{\mu}(z)-tz)=tF_{\mu}(z)-(t-1)z.
 $$
Applying $F_\mu$ to both sides of the previous equation we get
$$
(t+1)F_{\mu}(z)-tz=F_\mu(tF_{\mu}(z)-(t-1)z)
$$
or
$$
F_\mu(z)-z=F_\mu(tF_{\mu}(z)-(t-1)z)-tF_{\mu}(z)+(t-1)z. 
$$
Now, let $W_\mu(z)=F_{\mu^{\uplus t}}(z)=tF_\mu(z)-(t-1)z$ then 
$$
F_\mu(z)-z=F_\mu(W_\mu(z))-W_\mu(z). 
$$ 
Multiplied by $t$, this equation becomes 
$$
W_\mu(z)-z=W_\mu(W_\mu(z))-W_\mu(z). 
$$
This equation is exactly Equation (\ref{mastereq}) for $F=W_\mu$ which satisfies the assumptions of Lemma \ref{lema1}. So, $F_{\mu^{\uplus t}}$ is of the form $z-a_0+ib_0$ for some $a_0\in\mathbb{R}$ and $b_0 \geq 0$. This, in turn, implies that  $F_\mu(z)=z-a+ib$, where $a=\frac{a_0}{t}$ and $b=\frac{b_0}{t}$. If $b=0$, then $\mu=\delta_{a}$ and if $b>0$, $\mu=\gamma_{a,b}.$ 
\end{proof}
\begin{cor}The Boolean-to-free Bercovici-Pata bijection $\Lambda_B$ has no periodic points of order greater than one.
\end{cor}

\subsection{Proof of Proposition \ref{di}}
\begin{lem} \label{ap}
For any $m, n\in\mathbb{Z}, c \in \tor, t>0$ and $\mu \in \id (\putimes; \tor)_0$, the following are equivalent: \\ 
(1) $\mu^{\boxtimes_m t}$ exists in $\id (\putimes; \tor)_0$; \\
(2) $(\mu\putimes\delta_c)^{\boxtimes_n t}$ exists in $\id (\putimes; \tor)_0$. 
%If these equivalent conditions hold, then $\mu^{\boxtimes_m t}$ and $(\mu\putimes\delta_c)^{\boxtimes_n t}$ also belong to $\id (\putimes; \tor)_0$. 
\end{lem}
\begin{proof}
Note that  $\Sigma_{\mu \hutimes \delta_c}(z)=\frac{1}{c}\Sigma_{\mu}(\frac{z}{c})$ and $\Sigma_{\mu \boxtimes \delta_c}(z)=\frac{1}{c}\Sigma_\mu(z)$.
Hence $u=u(\mu,c,m,n)$ exists such that 
$$
\Sigma_{(\mu \hutimes\delta_c)^{\boxtimes_n  t}}(z)=\Sigma_{\mu \hutimes \delta_c}(z)^{t}=u\Sigma_{\mu ^{\boxtimes_m t}}\left(\frac{z}{c}\right)=cu \Sigma_{\mu ^{\boxtimes_m t}\hutimes \delta_c}(z)=
\Sigma_{\left(\mu ^{\boxtimes_m t}\hutimes \delta_c\right)\boxtimes  \delta_{cu} }(z).
$$
Thus $\left(\mu \putimes \delta_c\right)^{\boxtimes_n t}=\left(\mu ^{\boxtimes_m t}\putimes \delta_c\right)\boxtimes  \delta_{cu}$, which implies the equivalence between (1) and (2). 
\end{proof}
Also the following properties are useful.  
\begin{enumerate}[\rm(1)]
\item For any $\mu \in \id (\putimes; \tor)_0, m,n \in \mathbb{Z},  t\geq 1$, the measure $\left(\mu ^{\boxtimes_m t}\right)^{\boxtimes_n 1/ t} \in\id (\putimes; \tor)_0$ exists and equals 
$\mu\boxtimes\delta_c$ for some $c\in\tor$. 
\item For any $\mu \in \id (\putimes; \tor)_0, m,n \in \mathbb{Z},  t >0$, the measure $\left(\mu ^{\hutimes_m t}\right)^{\hutimes_n 1/t} \in\id (\putimes; \tor)_0$ exists and equals $\mu\putimes\delta_c$ for some $c\in\tor$.
\end{enumerate}

Let us go to the following which clearly implies Proposition \ref{di}. 
\begin{prop}\label{image}
$\mathbb{M}_t^{(n)}(\id (\putimes; \tor))=\mathbb{M}_t^{(0)}(\id (\putimes; \tor))$ for any $n$ and $t >0$.  
\end{prop}
\begin{proof}
Suppose $n \in \mathbb{Z}$, $t >0$ and $\mu \in \id(\putimes; \tor)_0$ and define $n':= (t+1)\arg_{(n)} m_1(\mu)$ with the notation that $\arg_{(n)} m_1(\mu)$ is the argument in $I_n$. Let $\nu \in \id(\putimes; \tor)_0$ and $c \in \tor$ be defined by 
$$
\nu:= \mathbb{M}^{(n)}_t(\mu)^{\hutimes_0(t+1)}= \left( \left(\mu ^{\boxtimes_n (t+1)}\right)^{\hutimes_{n'} 1/(t+1)}\right)^{\hutimes_0(t+1)} = \mu^{\boxtimes_n (t+1)} \putimes \delta_c. 
$$
From Lemma \ref{ap}, the measure $\nu^{\boxtimes_m 1/(t+1)}=\left(\mu ^{\boxtimes_n (t+1)} \putimes \delta_c \right)^{\boxtimes_m 1/(t+1)}$ exists in $\id(\putimes; \tor)_0$, where $m:= [(t+1) \arg_{(0)}m_1(\mu)]$. Note that $\arg_{(m)} m_1(\nu) \in I_m$ and $\frac{1}{t+1} \arg_{(m)} m_1(\nu) \in I_0$. Therefore
$$
 \mathbb{M}^{(0)}_t(\nu^{\boxtimes_m 1/(t+1)}) =  \nu^{\hutimes_m 1/(t+1)} = \mathbb{M}^{(n)}_t(\mu),  
$$
implying that $\mathbb{M}^{(0)}_t(\id(\putimes; \tor)_0) \supset \mathbb{M}^{(n)}_t(\id(\putimes; \tor)_0)$. The converse inclusion  is similar. 
\end{proof}

\section*{Acknowledgement}
The authors express sincere thanks to Professor Marek Bo\.zejko for informing them his conjecture and would also like to thank Dr.\ Serban Belinschi for useful discussions regarding this paper. They are grateful to Erwin Schr\"odinger Institute (ESI) in Vienna because part of this work was done during the authors' stay at ESI, within the program "Bialgebras in Free Probability" in February 2011. Professor V\'ictor P\'erez-Abreu made useful comments to improve the organization of this paper.

\end{document}